\documentclass[a4paper]{article}

\usepackage[top=2cm, bottom=2cm, left=1.8cm, right=1.8cm]{geometry}
\usepackage{amsfonts}
\usepackage{latexsym}
\usepackage{amsmath}
\usepackage{setspace}
\usepackage{comment}
\usepackage{bm}
\usepackage{graphics}
\usepackage{graphicx}
\usepackage{epstopdf}
\usepackage[caption=false]{subfig}
\usepackage{color}
\usepackage{algorithmicx}
\usepackage{algorithm}
\usepackage{algpseudocode}
\usepackage{amsopn}
\usepackage{amssymb}
\usepackage{mathtools}
\usepackage{amsthm}
\usepackage{booktabs}

\newcommand{\mini}{\mathop{\rm{minimize}}}
\newcommand{\subj}{{\mbox{subject to}}}

\newcommand{\A}{{\cal A}}
\newcommand{\B}{{\cal B}}
\newcommand{\C}{{\cal C}}
\newcommand{\D}{{\rm D}}

\newcommand{\N}{{\cal N}}
\newcommand{\I}{{\cal I}}
\newcommand{\J}{{\cal J}}
\newcommand{\U}{{\cal U}}

\newcommand{\F}{{\cal F}}
\newcommand{\R}{\mathbb{R}}

\renewcommand{\S}{\mathbb{S}}
\renewcommand{\P}{{\cal P}}

\theoremstyle{definition}
\newtheorem{definition}{Definition}
\newtheorem{theorem}{Theorem}

\newtheorem{lemma}{Lemma}
\newtheorem{remark}{Remark}

\begin{document}

\title{
A twice continuously differentiable penalty function for nonlinear semidefinite programming problems and its application
}

\author{Yuya Yamakawa\thanks{Graduate School of Management, Tokyo Metropolitan University, 1-1 Minami-Osawa, Hachioji-shi, Tokyo 192-0397, Japan, E-mail: yuya@tmu.ac.jp}}

\date{}

\maketitle

\begin{abstract}
This paper presents a twice continuously differentiable penalty function for nonlinear semidefinite programming problems. In some optimization methods, such as penalty methods and augmented Lagrangian methods, their convergence property can be ensured by incorporating a penalty function into them, and hence several types of penalty functions have been proposed. In particular, these functions are designed to apply optimization methods to find first-order stationary points. Meanwhile, in recent years, second-order sequential optimality, such as Approximate Karush-Kuhn-Tucker2 (AKKT2) and Complementarity AKKT2 (CAKKT2) conditions, has been introduced, and the development of methods for such second-order stationary points would be required in future research. However, existing well-known penalty functions have low compatibility with such methods because they are not twice continuously differentiable. In contrast, the proposed function is expected to have a high affinity for methods to find second-order stationary points. To verify the high affinity, we also present a practical penalty method to find points that satisfy the AKKT and CAKKT conditions by exploiting the proposed function and show their convergence properties.
\end{abstract}

\section{Introduction} \label{sec:intro}
This study proposes a twice continuously differentiable penalty function related to the following nonlinear semidefinite programming problem (NSDP):
\begin{align}
\begin{aligned} \label{NSDP}
& \mini_{x \in \R^{n}} & & f(x)
\\ 
& \subj & & g(x) = 0, ~ G(x) \succeq O,
\end{aligned}
\end{align}
where the functions $f \colon \R^{n} \rightarrow \R$, $g \colon \R^{n} \rightarrow \R^{m}$, and $G \colon \R^{n} \rightarrow \S^{d}$ are twice continuously differentiable, and $\R^{n}$ denotes the $n$-dimensional Euclidean space, and $\S^{d}$ represents the set of $d \times d$ real symmetric matrices, and $\S_{++}^{d} \, (\S_{+}^{d})$ denotes the set of symmetric positive (semi)definite matrices in $\S^{d}$, and we often denote $M \in \S_{++}^{d}$ and $M \in \S_{+}^{d}$ by $M \succ O$ and $M \succeq O$, respectively, for simplicity. 
\par
NSDPs have been actively studied since the 2000s, and many optimization methods have been proposed for solving them~\cite{CoRa04,FaNoAp02,LeMo02,OkYaFu23,SuSuZh08,WuLuDiCh13,YaOk22,YaYa14,YaYa15IPmethod,YaYa12local,YaYaHa12global,YaYaHa21TRmethod}. Moreover, since they have a higher ability to formulate real problems as mathematical models than nonlinear programming problems (NLPs), the range of their applications is expanding as studies on NSDPs develop. In recent years, sequential optimality has been proposed as genuine necessary optimality for degenerate NSDPs, which satisfy no constraint qualification (CQ), and studies on optimization methods to find points that satisfy those optimality conditions have also been conducted actively.
\par
Sequential optimality serves as necessary optimality for degenerate NSDPs and has several varieties, such as Approximate Karush-Kuhn-Tucker (AKKT), Complementarity AKKT (CAKKT), AKKT2, and CAKKT2 conditions~\cite{AnFuHaSaSe21,AnHaVi20,FuHaMi23,LiYaFu25}. The AKKT and CAKKT conditions are the first-order necessary optimality, and some related optimization methods to find points that satisfy them (i.e., first-order stationary points) have been proposed so far, such as augmented Lagrangian methods~\cite{AnFuHaSaSe21,AnHaVi20} and stabilized sequential quadratic semidefinite programming (SQSDP) methods~\cite{OkYaFu23,YaOk22}. On the other hand, the AKKT2 and CAKKT2 conditions are the second-order necessary optimality, but little research has been done on optimization methods to find points that satisfy them (i.e., second-order stationary points). To the best of the author's knowledge, although the only method regarding AKKT2 and CAKKT2 is the penalty method presented in~\cite{LiYaFu25}, it requires solving a subproblem, which minimizes a penalty function belonging to the H\"{o}lder space $C^{1,1}$, at each iteration and finding its stationary point satisfying some additional conditions similar to its second-order optimality. However, Li et al.~\cite{LiYaFu25} mentioned that the solvability of the subproblems (i.e., a way of finding such a stationary point) is an open question, and hence, the penalty method has a practical issue in the sense that its implementability is not guaranteed.
\par
In the field of nonlinear optimization, augmented Lagrangian methods and sequential quadratic programming (SQP)-type methods are known to be very useful and are implemented in many solvers. Indeed, some variants based on them have been proposed to find points that satisfy the AKKT and CAKKT conditions as stated above. These methods often utilize penalty functions to measure the distance between the optimal solution and an iteration point. For example, Correa and Ram\'{i}rez~\cite{CoRa04} proposed an SQP-type method, employing a penalty function belonging to the class $C$. Moreover, Andreani, Haeser, and Viana~\cite{AnHaVi20} and Yamakawa and Okuno~\cite{YaOk22} proposed an AL method and a stabilized SQSDP method, respectively, and they used essentially the same functions included in $C^{1,1}$. These functions play a crucial role in the convergence analysis of these methods for first-order stationary points. However, they are not twice continuously differentiable, and hence this fact brings practical challenges when considering the development of optimization methods for second-order stationary points. Indeed, as aforementioned above, Li, Yamakawa, and Fukuda~\cite{LiYaFu25} proposed the penalty method, equipped with the penalty function belonging to $C^{1,1}$, regarding the AKKT2 and the CAKKT2 conditions, but the method has an issue regarding the solvability of the subproblems. The issue arises from the fact that the penalty function is not twice continuously differentiable.
\par
In this study, we propose a twice continuously differentiable penalty function for problem~\eqref{NSDP}. Since the proposed penalty function has a higher-order differentiability compared to the existing penalty functions, it is expected to have a high affinity for optimization methods to find second-order stationary points and to bring better properties than the existing penalty functions in practical aspects. Concerning these facts, we provide some properties associated with the proposed penalty function. Moreover, to show that the proposed function has practical applications in the development of optimization methods for second-order stationary points, we present a penalty method equipped with the proposed one and prove that it globally converges to points that satisfy the AKKT2 and the CAKKT2 conditions. A remarkable point of this study is that by designing the penalty method using the proposed function, we can overcome the issue regarding the solvability of the subproblems as seen in~\cite{LiYaFu25} because the subproblems can be solved by the existing methods for unconstrained optimization of twice continuously differentiable functions. Moreover, the proposed method has novel points in the sense that it is implementable and can find second-order stationary points.
\par
The paper is organized as follows. Section~\ref{sec:optimality} introduces mathematical symbols and important concepts associated with~\eqref{NSDP}. Section~\ref{sec:penalty} proposes a twice continuously differentiable penalty function for~\eqref{NSDP} and provides its properties. In Section~\ref{sec:methods}, we present a penalty method and show that it converges to points satisfying second-order optimality conditions. Finally, conclusions are presented in Section~\ref{sec:conclusion}.

\section{Preliminaries} \label{sec:optimality}
This section is divided into two parts. The former part introduces mathematical notation, and the latter part provides several types of optimality conditions for problem~\eqref{NSDP}.

\subsection{Mathematical notation}
Let $\mathbb{N}$ be the set of natural numbers (positive integers). Let $p \in \mathbb{N}$ and $q \in \mathbb{N}$. We denote by $\R^{p \times q}$ the set of matrices with $p$-rows and $q$-columns. The symbol $\S_{--}^{p} \, (\S_{-}^{p})$ represents the set of symmetric negative (semi)definite matrices. Let $w \in \R^{p}$ and $W \in \R^{p \times q}$. The $i$-th element of $w$ is represented by $w_{i}$ or $[w]_{i}$, and the $(i,j)$-th entry of $W$ is written as $W_{ij}$ or $[W]_{ij}$. The transposition of $W$ is represented as $W^{\top} \in \R^{q \times p}$. We denote by $W^{\dagger} \in \R^{q \times p}$ the Moore-Penrose generalized inverse of $W$. The trace of a square matrix $W$ is denoted by ${\rm tr}(M)$. For any $V \in \R^{p \times q}$ and $W \in \R^{p \times q}$, the inner product of $V$ and $W$ is defined as $\left\langle V, W \right\rangle \coloneqq {\rm tr}(V^{\top}W)$. If $q = 1$ holds, we denote by $\left\langle v, w \right\rangle$ the inner product of $v \in \R^{p}$ and $w \in \R^{p}$. The Hadamard product of $V$ and $W$ is denoted by 
\begin{align*}
V \circ W \coloneqq
\begin{bmatrix}
V_{11} W_{11} & \cdots & V_{1q} W_{1q}
\\
\vdots & \ddots & \vdots
\\
V_{p1} W_{p1} & \cdots & V_{pq} W_{pq}
\end{bmatrix}
.
\end{align*}
If $V$ and $W$ are symmetric matrices, we define 
\begin{align*}
V \circledcirc W \coloneqq \frac{1}{2}(VW + WV).
\end{align*}
We denote by $I$ the identity matrix, where its dimension is defined by the context. For any $w \in \R^{p}$, we denote by $\Vert w \Vert$ the Euclidean norm of $w$, that is, $\Vert w \Vert \coloneqq \sqrt{\left\langle w, w \right\rangle}$. For any $z \in \R^{p}$ and $r > 0$, we define 
\begin{align*}
B(z, r) \coloneqq \{ \zeta \in \R^{p}; \Vert \zeta - z \Vert \leq r \}.
\end{align*}
For any $W \in \R^{p \times q}$, we denote by $\Vert W \Vert_{{\rm F}}$ the Frobenius norm of $W$, namely, $\Vert W \Vert_{{\rm F}} \coloneqq \sqrt{\left\langle W, W \right\rangle}$. 
For any $r_{1}, \ldots, r_{\ell} \in \R$, the diagonal matrix whose $(i, i)$-th entry is equal to $r_{i}$ is written by ${\rm diag} ( r_{1}, \ldots, r_{\ell} )$. Let $X \in \S^{\ell}$ be a matrix with the following orthogonal diagonalization:
\begin{align*}
X = P 
\begin{bmatrix}
\lambda_{1}^{P}(X) & & O
\\
& \ddots &
\\
O & & \lambda_{\ell}^{P}(X)
\end{bmatrix}
P^{\top}, \quad P P^{\top} = P^{\top} P = I,
\end{align*}
where $\lambda_{1}^{P}(X), \ldots, \lambda_{\ell}^{P}(X)$ stand for the eigenvalues of $X$. In particular, we omit the superscript $P$ when considering a diagonalization that places the eigenvalues of $X$ in descending order $\lambda_{1}(X) \geq \cdots \geq \lambda_{\ell}(X)$. Moreover, we use the following notation:
\begin{align*}
\lambda^{P}(X) \coloneqq
\begin{bmatrix}
\lambda_{1}^{P}(X)
\\
\vdots
\\
\lambda_{\ell}^{P}(X)
\end{bmatrix}
, \quad \lambda(X) \coloneqq
\begin{bmatrix}
\lambda_{1}(X)
\\
\vdots
\\
\lambda_{\ell}(X)
\end{bmatrix}
.
\end{align*}
The matrices $[X]_{+}$ and $|X|$ are defined by
\begin{align*}
[ X ]_{+} \coloneqq P
\begin{bmatrix}
[\lambda_{1}^{P}(X)]_{+} & & O
\\
& \ddots &
\\
O & & [\lambda_{\ell}^{P}(X)]_{+}
\end{bmatrix}
P^{\top}, \quad | X | \coloneqq P
\begin{bmatrix}
|\lambda_{1}^{P}(X)| & & O
\\
& \ddots &
\\
O & & |\lambda_{\ell}^{P}(X)|
\end{bmatrix}
P^{\top},
\end{align*}
respectively, where $[r]_{+} \coloneqq \max \{ r, 0 \}$ for any $r \in \R$. Let $\varphi$ be a function from $\R^{p}$ to $\R$. We denote by $\nabla \varphi(w)$ (or $\nabla_{w} \varphi(w)$) the gradient of $\varphi$ at $w \in \R^{p}$. We write $\nabla^{2} \varphi(w)$ (or $\nabla_{ww}^{2} \varphi(w)$) for the Hessian of $\varphi$ at $w \in \R^{p}$. Let $\psi_{1}, \ldots, \psi_{q}$ be functions from $\R^{p}$ to $\R$, and let $\psi(w) \coloneqq [\psi_{1}(w) \, \cdots \, \psi_{q}(w)]^{\top} \in \R^{q}$ for any $w \in \R^{p}$. We define $\nabla \psi(w) \coloneqq [\nabla \psi_{1}(w) \, \cdots \, \nabla \psi_{q}(w)] \in \R^{p \times q}$ for any $w \in \R^{p}$. Let $\Phi$ be a function from ${\cal X}$ to ${\cal Y}$, where ${\cal X}$ and ${\cal Y}$ are inner product spaces.
The derivative of $\Phi$ at $\xi \in {\cal X}$ is denoted by $\D \Phi(\xi)$. We denote by $\D \Phi(\xi)^{\ast}$ the adjoint operator of $\D \Phi(\xi)$. If $\Phi$ is a linear operator, its norm is defined by 
\begin{align*}
\Vert \Phi \Vert \coloneqq \sup \{ \Vert \Phi(\xi) \Vert; \Vert \xi \Vert \leq 1 \}.
\end{align*}
For a set $S \subset \mathbb{N}$, the cardinality of $S$ is represented by $|S|$.

\subsection{Optimality conditions}
Let $g_{1}, \ldots, g_{m}$ be functions from $\R^{n}$ to $\R$ satisfying
\begin{align*}
g(x) = [g_{1}(x) ~ \cdots ~ g_{m}(x)]^{\top} \quad \forall x \in \R^{n}. 
\end{align*} 
We denote by $\mathcal{S}$ the set of feasible points of problem~\eqref{NSDP}, namely,
\begin{align*}
\mathcal{S} \coloneqq \{ x \in \R^{n}; \, g(x) = 0, \, G(x) \succeq O \}. 
\end{align*}
Let us define the Lagrange function $L \colon \R^{n} \times \R^{m} \times \S^{d} \to \R$ as
\begin{eqnarray*}
L(x, y, Z) \coloneqq f(x) - \langle g(x), y \rangle - \langle G(x), Z \rangle,
\end{eqnarray*}
where $y$ and $Z$ are Lagrange multipliers for $g(x) = 0$ and $G(x) \succeq 0$, respectively. Then, the gradient of $L$ at $(x, y, Z) \in \R^{n} \times \R^{m} \times \S^{d}$ with respect to $x$ is given by 
\begin{eqnarray*}
\nabla _{x} L(x, y, Z) = \nabla f(x) - \nabla g(x) y - \D G(x)^{\ast} Z.
\end{eqnarray*}
We introduce the Karush-Kuhn-Tucker (KKT) conditions for~\eqref{NSDP}. 

\begin{definition}
We say that the KKT conditions hold at a feasible point $\bar{x} \in \R^{n}$ if there exist $\bar{y} \in \R^{m}$ and $\bar{Z} \in \S_{+}^{d}$ such that
\begin{eqnarray*} 
\displaystyle \nabla_{x} L(\bar{x}, \bar{y}, \bar{Z}) = 0, \quad \langle G(\bar{x}), \bar{Z} \rangle = 0.
\end{eqnarray*}
Moreover, we call a point $\bar{x}$ satisfying the KKT conditions a KKT point and call $(\bar{x}, \bar{y}, \bar{Z})$ a KKT triplet.
\end{definition}
\noindent
The KKT conditions are the first-order optimality for problem~\eqref{NSDP} under some CQ. We provide the Robinson CQ (RCQ) for problem~\eqref{NSDP}.
\begin{definition}
We say that RCQ holds at a feasible point $\bar{x} \in \mathcal{S}$ if
\begin{align*}
0 \in {\rm int} \left(
\begin{bmatrix}
g(\bar{x})
\\
G(\bar{x})
\end{bmatrix}
+
\begin{bmatrix}
\nabla g(\bar{x})^{\top}
\\
\D G(\bar{x})
\end{bmatrix}
\R^{n}
-
\begin{bmatrix}
\{ 0 \}
\\
\S_{+}^{d}
\end{bmatrix}
\right),
\end{align*}
where ${\rm int}(U)$ denotes the topologicall interior of a set $U \in \R^{m} \times \S^{d}$.
\end{definition}
\par
Next, we provide the definitions of the AKKT and CAKKT conditions for NSDP \eqref{NSDP}. These conditions have been introduced by~\cite{AnFuHaSaSe24,AnHaVi20}. 
\begin{definition} \label{AKKT_def}
We say that the AKKT conditions hold at a feasible point $\bar{x} \in \mathcal{S}$ if there exist sequences $\{ x_{k} \} \subset \R^{n}, ~ \{ y_{k} \} \subset \R^{m}$, and $\{ Z_{k} \} \subset \S^{d}_{+}$ such that 
\begin{gather*}
\displaystyle \lim_{j \to \infty} x_{k} = \bar{x}, \quad \displaystyle \lim_{k \to \infty} \nabla_{x} L(x_{k}, y_{k}, Z_{k}) = 0,
\\
\lambda_{j}^{P}(G(\bar{x})) > 0 \quad \Longrightarrow \quad \exists k_{j} \in \mathbb{N}, ~~ \forall k \geq k_{j}, ~~ \lambda_{j}^{P_{k}}(Z_{k}) = 0, 
\end{gather*}
where $j \in \{1, \ldots, d\}$ is arbitrary, and $P$ and $P_{k} ~ (k \in \mathbb{N})$ are orthogonal matrices such that 
\begin{align*}
\lim_{k \to \infty} P_{k} = P, \quad G(\bar{x}) = P {\rm diag} ( \lambda_{1}^{P}(G(\bar{x})), \ldots, \lambda_{d}^{P}(G(\bar{x})) ) P^{\top}, \quad Z_{k} = P_{k} {\rm diag} ( \lambda_{1}^{P_{k}}(Z_{k}), \ldots, \lambda_{d}^{P_{k}}(Z_{k}) ) P_{k}^{\top}. 
\end{align*}
Moreover, we call a point $\bar{x}$ satisfying the AKKT conditions an AKKT point and call a sequence $\{ (x_{k}, y_{k}, Z_{k}) \}$ an AKKT sequence corresponding to $\bar{x}$.
\end{definition}
\begin{definition} \label{CAKKT_def}
We say that the CAKKT conditions hold at a feasible point $\bar{x} \in \mathcal{S}$ if there exist sequences $\{ x_{k} \} \subset \R^{n}$, $\{ y_{k} \} \subset \R^{m}$, and $\{ Z_{k} \} \subset \S^{d}_{+}$ such that 
\begin{align*}
\displaystyle \lim_{k \to \infty} x_{k} = \bar{x}, \quad \displaystyle \lim_{k \to \infty} \nabla_{x} L(x_{k}, y_{k}, Z_{k}) = 0, \quad \lim_{k \to \infty} G(x_{k}) \circledcirc Z_{k} = O.
\end{align*}
Moreover, we call a point $\bar{x}$ satisfying the CAKKT conditions a CAKKT point and call a sequence $\{ (x_{k}, y_{k}, Z_{k}) \}$ a CAKKT sequence corresponding to $\bar{x}$.
\end{definition}
\noindent
The KKT, the AKKT, and the CAKKT conditions are known as first-order optimality for problem~\eqref{NSDP}. In particular, the AKKT and CAKKT conditions are referred to as sequential optimality and are significantly different from the KKT conditions in the sense that they serve as optimality conditions regardless of whether CQs hold or not. (Conversely, the KKT conditions are meaningless as necessary optimality conditions unless some CQ, such as RCQ, exists.) Moreover, it is generally known that CAKKT implies AKKT. For details, see~\cite{AnFuHaSaSe24,AnHaVi20}.
\par
In the remainder of this section, we introduce second-order necessary conditions for problem~\eqref{NSDP}. To this end, we prepare some terminology. For any $x \in \R^{n}$, let us define
\begin{align*}
G_{i}(x) \coloneqq \frac{\partial}{\partial x_{i}} G(x) \quad ( \, i = 1, \ldots, n \, ).
\end{align*}
For all $(x,Z) \in \R^{n} \times \S^{d}$, the sigma-term is defined as
\begin{align*}
\sigma(x, Z) \coloneqq \left[ 2 \langle Z, G_{i}(x) G(x)^{\dagger} G_{j}(x) \rangle \right]_{ij} \in \S^{d}.
\end{align*} 
The critical cone at a point $x$ is defined by
\begin{align*}
C(x) \coloneqq \left\{ h \in \R^{n}; \, \langle \nabla f(x), h \rangle = 0, \, \nabla g(x)^{\top} h = 0, \, \D G(x) h \in T_{\S_{+}^{d}}(G(x)) \right\},
\end{align*}
where
\begin{align*}
T_{\S_{+}^{d}}(G(x)) \coloneqq \left\{ D \in \S^{d}; 
\begin{array}{c}
\displaystyle \exists \{ D_{k} \} \subset \S^{d}, \, \exists \{ \alpha_{k} \} \subset \R, \, \lim_{k \to \infty} D_{k} = D, \, \lim_{k \to \infty} \alpha_{k} = 0, 
\\
G(x) + \alpha_{k} D_{k} \in \S_{+}^{d}, ~~ \alpha_{k} > 0 \quad \forall k \in \mathbb{N}
\end{array}
\right\}.
\end{align*}
Now, we define a basic second-order necessary condition (BSONC) for problem~\eqref{NSDP}.
\begin{definition} \label{def:BSONC}
We say that BSONC holds at a feasible point $\bar{x} \in \mathcal{S}$ if for any $h \in C(\bar{x})$, there exist $\bar{y}_{h} \in \R^{m}$ and $\bar{Z}_{h} \in \S_{+}^{d}$ such that $(\bar{x}, \bar{y}_{h}, \bar{Z}_{h})$ is a KKT triplet, and
\begin{align*}
h^{\top} \left( \nabla_{xx}^{2} L(\bar{x}, \bar{y}_{h}, \bar{Z}_{h}) + \sigma(\bar{x}, \bar{Z}_{h}) \right) h \geq 0.
\end{align*}
\end{definition} 
\noindent
Note that the Lagrange multiplier pair $(\bar{y}_{h}, \bar{Z}_{h})$ described in Definition~\ref{def:BSONC} depends on $h \in C(\bar{x})$. As mentioned in~\cite{FuHaMi23,LiYaFu25}, this fact makes it difficult to verify whether BSONC holds or not. To overcome this drawback, the weak second-order necessary condition (WSONC) was introduced in~\cite{FuHaMi23}. Although WSONC overcomes the drawback of BSONC, it requires several additional conditions, such as RCQ and the weak constant rank property, to be a necessary condition. For this reason, the AKKT2 and the CAKKT2 conditions were proposed in~\cite{LiYaFu25}. Of course, they are second-order necessary optimality conditions that do not require any additional conditions, unlike WSONC. To introduce these two conditions, we prepare the following lemma.
\begin{lemma} \cite[Example~3.140]{BoSh00} \label{lem:BoSh00}
Suppose that $G(\bar{x}) \succeq O$ holds. Let $\B$ be a set defined by $\B \coloneqq \{ j \in \mathbb{N}; \lambda_{j}(G(\bar{x})) = 0 \}$, and let $U_{\B}$ be a matrix with orthonormal columns that span ${\rm Ker}(G(\bar{x}))$. Then, there exist a neighborhood $\N$ of $G(\bar{x})$ and an analytic matrix-valued function $\U_{\B} \colon \N \to \R^{d \times |\B|}$ such that $\U_{\B}(G(\bar{x})) = U_{\B}$ and for any $x \in \R^{n}$ satisfying $G(x) \in \N$, the columns of $\U_{\B}(G(x))$ form an orthonormal basis of a space spanned by the eigenvectors associated with the $|\B|$ smallest eigenvalues of $G(x)$.
\end{lemma}
\noindent
Moreover, we define the perturbed critical subspace $S(x_{k}, \bar{x}) \subset \R^{n}$. Let $\bar{x} \in \R^{n}$ be a feasible point and let $\{ x_{k} \} \subset \R^{n}$ be a sequence converging to $\bar{x}$. Since $G(\bar{x}) \succeq O$ holds, Lemma~\ref{lem:BoSh00} ensures the existence of a neighborhood $\N$ of $G(\bar{x})$ and a analytic function $\U_{\B} \colon \N \to \R^{d \times |\B|}$ as described in Lemma~\ref{lem:BoSh00}, where $\B$ is the set defined in Lemma~\ref{lem:BoSh00}. From the continuity of $G$, there exists $\bar{n} \in \mathbb{N}$ such that $G(x_{k}) \in \N$ for all $k \geq \bar{n}$. Then, the perturbed critical subspace $S(x_{k}, \bar{x})$ is defined as
\begin{align*}
S(x_{k}, \bar{x}) \coloneqq \left\{ h \in \R^{n}; \, \nabla g(x_{k})^{\top} h = 0, \, \U_{\B}(G(x_{k}))^{\top} (\D G(x_{k})h) \U_{\B}(G(x_{k})) = O \right\} \quad \forall k \geq \bar{n}.
\end{align*}

\noindent
By using the above notation, the AKKT2 and the CAKKT2 conditions are defined as follows:
\begin{definition} \label{def:AKKT2}
We say that the AKKT2 conditions hold at a feasible point $\bar{x} \in \mathcal{S}$ if there exist $\{ (x_{k}, y_{k}, Z_{k}, \varepsilon_{k}) \} \subset \R^{n} \times \R^{m} \times \S_{+}^{d} \times (0, \infty)$ and $\bar{n} \in \mathbb{N}$ such that $x_{k} \to \bar{x}$ and $\varepsilon_{k} \to 0$ as $k \to \infty$, $\{ (x_{k}, y_{k}, Z_{k}) \}$ is an AKKT sequence corresponding to $\bar{x}$, and
\begin{align*}
\forall k \geq \bar{n}, \quad \forall h \in S(x_{k}, \bar{x}), \quad h^{\top} \left( \nabla_{xx}^{2} L(x_{k}, y_{k}, Z_{k}) + \sigma(x_{k}, Z_{k}) \right) h \geq - \varepsilon_{k} \Vert h \Vert^{2}.
\end{align*}
Moreover, we call $\bar{x}$ an AKKT2 point and call $\{ (x_{k}, y_{k}, Z_{k}, \varepsilon_{k}) \}$ an AKKT2 sequence corresponding to $\bar{x}$.
\end{definition}

\begin{definition} \label{def:CAKKT2}
We say that the CAKKT2 conditions hold at a feasible point $\bar{x} \in \mathcal{S}$ if there exist $\{ (x_{k}, y_{k}, Z_{k}, \varepsilon_{k}) \} \subset \R^{n} \times \R^{m} \times \S_{+}^{d} \times (0, \infty)$ and $\bar{n} \in \mathbb{N}$ such that $x_{k} \to \bar{x}$ and $\varepsilon_{k} \to 0$ as $k \to \infty$, $\{ (x_{k}, y_{k}, Z_{k}) \}$ is a CAKKT sequence corresponding to $\bar{x}$, and
\begin{align*}
\forall k \geq \bar{n}, \quad \forall h \in S(x_{k}, \bar{x}), \quad h^{\top} \left( \nabla_{xx}^{2} L(x_{k}, y_{k}, Z_{k}) + \sigma(x_{k}, Z_{k}) \right) h \geq - \varepsilon_{k} \Vert h \Vert^{2}.
\end{align*}
Moreover, we call $\bar{x}$ a CAKKT2 point and call $\{ (x_{k}, y_{k}, Z_{k}, \varepsilon_{k}) \}$ a CAKKT2 sequence corresponding to $\bar{x}$.
\end{definition}

\noindent
As aforementioned, AKKT2 and CAKKT2 are second-order necessary conditions without requiring any assumptions. These facts are ensured by~\cite[Theorems~4.5 and 4.8]{LiYaFu25}. Moreover, they are referred to as second-order sequential optimality.

\section{Twice continuously differentiable penalty function} \label{sec:penalty}
The purpose of this section is to propose a penalty function associated with problem~\eqref{NSDP}. The proposed function is twice continuously differentiable and is expected to play a crucial role in designing optimization methods to find second-order stationary points for~\eqref{NSDP}.
\par
We begin by reviewing the existing penalty functions associated with problem~\eqref{NSDP}. The following penalty function is often used in SQP-type methods, such as~\cite{CoRa04,LiZh19}:
\begin{align*}
\mathcal{P}_{1}(x; \sigma) \coloneqq f(x) + \sigma \Vert g(x) \Vert + \sigma [\lambda_{\max}(-G(x))]_{+},
\end{align*}
where $\sigma > 0$ is a given parameter. Although the above function is not differentiable clearly, it is known that its directional derivative exists if $x$ is feasible. For details, see~\cite[Lemma~6]{CoRa04}.
\par
Next, we introduce the following penalty function:
\begin{align*}
\mathcal{P}_{2}(x; v, M, \tau) \coloneqq f(x) + \frac{\tau}{2} \left\Vert \frac{1}{\tau} v - g(x) \right\Vert^{2} + \frac{\tau}{2} \left\Vert \left[ \frac{1}{\tau} M - G(x) \right]_{+} \right\Vert_{{\rm F}}^{2},
\end{align*}
where $v \in \R^{m}$, $M \in \S^{d}$, and $\tau > 0$. The above function is essentially equal to the augmented Lagrangian, and it is utilized in several existing methods, such as~\cite{AnHaVi20,OkYaFu23,SuSuZh08,YaOk22}. Moreover, this function is continuously differentiable, and its gradient is Lipschitz continuous, and thus it belongs to the H\"{o}lder space $C^{1,1}$.
\par
To measure the distance between $x \in \R^{n}$ and the feasible set of problem~\eqref{NSDP}, the following penalty function is often used:
\begin{align*}
\mathcal{P}_{3}(x) \coloneqq \frac{1}{2} \Vert g(x) \Vert^{2} + \frac{1}{2} \Vert [-G(x)]_{+} \Vert_{{\rm F}}^{2}.
\end{align*}
Note that this function also belongs to the H\"{o}lder space $C^{1,1}$ and can be found in the existing studies regarding augmented Lagrangian methods~\cite{AnFuHaSaSe21,AnHaVi20} and stabilized SQSDP methods~\cite{OkYaFu23,YaOk22}. 
\par
The functions $\mathcal{P}_{1}$, $\mathcal{P}_{2}$, and $\mathcal{P}_{3}$ are utilized in existing optimization methods for finding first-order stationary points and play a crucial role in them. However, these functions are not twice continuously differentiable, and thus they are not suitable for developing optimization methods for second-order stationary points, as mentioned in Section~\ref{sec:intro}. Therefore, this study proposes the following twice continuously differentiable penalty function:
\begin{align*}
F(x; v, M, \rho, \sigma, \tau) \coloneqq \rho f(x) + \frac{\sigma \tau}{2} \left\Vert \frac{1}{\tau} v - g(x) \right\Vert^{2} + \frac{\sigma \tau}{4} {\rm tr} \left( \left[ \frac{1}{\tau} M - G(x) \right]_{+}^{4} \right),
\end{align*}
where $v \in \R^{m}$, $M \in \S^{d}$, $\rho > 0$, $\sigma > 0$, and $\tau > 0$ are parameters.
\par
From now on, we consider the differentiability of the proposed function. Since the first and second terms are clearly continuously differentiable, we show that the third term is continuously differentiable.

\begin{lemma} \label{lem:derivative_thirdterm}
Let $M \in \S^{d}$ and $\tau > 0$ be given, and let $\varphi \colon \R^{n} \to \R$ be defined by $\varphi(x) \coloneqq {\rm tr}( [\frac{1}{\tau}M - G(x)]_{+}^{4} )$ for any $x \in \R^{n}$. Then, the function $\varphi$ is continuously differentiable on $\R^{n}$, and its gradient is given by
\begin{align*}
\nabla \varphi (x) = - 4 \D G(x)^{\ast} \left[ \frac{1}{\tau} M - G(x) \right]_{+}^{3}
\end{align*}
for any $x \in \R^{n}$.
\end{lemma}

\begin{proof}
Let us define $\phi \colon \R^{d} \to \R$ as $\phi(w) \coloneqq [w_{1}]_{+}^{4} + \cdots + [w_{d}]_{+}^{4}$ for any $w \in \R^{d}$. We take $W \in \S^{d}$ arbitrarily. Then, note that
\begin{align*}
\psi(W) \coloneqq \phi(\lambda(W)) = \sum_{j=1}^{d} [\lambda_{j}(W)]_{+}^{4} = {\rm tr} \left( [ W ]_{+}^{4} \right).
\end{align*}
We show that $\psi \colon \S^{d} \to \R$ is continuously differentiable on $\S^{d}$. Recall that $\phi$ is continuously differentiable on $\R^{d}$, and its gradient is given by $\nabla \phi(w) = 4 [[w_{1}]_{+}^{3}, \ldots, [w_{d}]_{+}^{3}]^{\top}$ for each $w \in \R^{d}$. It then follows from~\cite[Corollary~3.2]{Le96} that $\psi$ is continuously differentiable on $\S^{d}$, and its derivative is given by $\D \psi(W) = 4 [W]_{+}^{3}$. By these facts and the chain rule of differentiation, we can easily verify that the assertion is true.
\end{proof}

\par
Next, we show that the proposed function is twice continuously differentiable. To this end, we define the function $\mathcal{Q} \colon \S^{d} \to \S^{d}$ as \begin{align}
\mathcal{Q}(X) \coloneqq [X]_{+}^{3} \label{def:funcQ}
\end{align}
for any $X \in \S^{d}$, and consider the differentiability of $\mathcal{Q}$.

\begin{lemma} \label{lem:derivative_Q}
The function $\mathcal{Q}$ is differentiable on $\S^{d}$, and the derivative of $\mathcal{Q}$ at $X$ is given as follows: if $X$ is a zero matrix, then $\D \mathcal{Q} (X) H = O$ for all $H \in \S^{d}$; if $X$ is a nonzero matrix that has the following decomposition
\begin{align*}
\begin{gathered}
X = 
\begin{bmatrix}
P_{\I} & P_{\J}
\end{bmatrix}
\begin{bmatrix}
\Lambda & O
\\
O & O
\end{bmatrix}
\begin{bmatrix}
P_{\I}^{\top}
\\
P_{\J}^{\top}
\end{bmatrix}
, ~~ P_{\I} \in \R^{d \times \ell}, ~~ P_{\J} \in \R^{d \times (d-\ell)},~~ \Lambda \coloneqq {\rm diag}(\mu_{1}, \ldots, \mu_{\ell}), ~~ \mu_{j} \not = 0 ~~ \forall j \in \{1, \ldots, \ell \},
\end{gathered}
\end{align*}
then,
\begin{align*}
{\rm D} {\cal Q}(X) H = 
\begin{bmatrix}
P_{\I} & P_{\J}
\end{bmatrix}
\begin{bmatrix}
A \circ P_{\I}^{\top} H P_{\I} & B P_{\I}^{\top} H P_{\J}
\\
P_{\J}^{\top} H P_{\I} B & O
\end{bmatrix}
\begin{bmatrix}
P_{\I}^{\top}
\\
P_{\J}^{\top}
\end{bmatrix}
\end{align*}
for all $H \in \S^{d}$, where $A \in \S^{\ell}$ and $B \in \S^{\ell}$ are defined by
\begin{align*}
A \coloneqq \left[ \frac{(|\mu_{i}|^{3} + \mu_{i}^{3} + |\mu_{j}|^{3} + \mu_{j}^{3})(\mu_{i}^{2} + \mu_{i}\mu_{j} + \mu_{j}^{2})}{2(|\mu_{i}|^{3} + |\mu_{j}|^{3})} \right]_{ij}, \quad B \coloneqq \frac{1}{2} ( |\Lambda| + \Lambda ) \Lambda.
\end{align*}
Note that if $X$ is nonsingular, the parts associated with $P_{\J}$ should be ignored.
\end{lemma}

\begin{proof}
To begin with, it can be easily verified that $\mathcal{Q}$ is differentiable at $X=O$. Moreover, the derivative of $\mathcal{Q}$ at $X=O$ satisfies that $\D \mathcal{Q}(X) H = O$ for any $H \in \S^{d}$. Therefore, we discuss the case where $X$ is a nonzero matrix. Note that $[X]_{+} = \frac{1}{2}(X + |X|)$. Then, we have $[X]_{+}^{3} = \frac{1}{8}(4 X^{3} + 4 |X|^{3}) = \frac{1}{2}(X^{3} + |X|^{3})$ for any $X \in \S^{d}$. In what follows, we consider the derivative of the function $\Phi \colon \S^{d} \to \S^{d}$ defined by $\Phi(X) \coloneqq |X|^3$. Let $X \in \mathbb{S}^{d} \backslash \{ O \}$ be arbitrary matrix. Without loss of generality, we suppose that a diagonal matrix $D$ and an orthogonal matrix $P$ satisfy
\begin{align} \label{eq:X_PDP}
\begin{gathered}
X = P D P^{\top}, \quad P = 
\begin{bmatrix}
P_{\I} & P_{\J}
\end{bmatrix}
, \quad D =
\begin{bmatrix}
\Lambda & O
\\
O & O
\end{bmatrix}
, \quad | \Lambda | \in \S_{++}^{\ell}, \quad P_{\I} \in \R^{d \times \ell}, \quad P_{\J} \in \R^{d \times (d - \ell)},
\end{gathered}
\end{align}
where if $X$ is nonsingular, we should consider $P = P_{\I}$ and $D = \Lambda$. As described below, the following proof also holds in such a case. 
Let us take $H \in \mathbb{S}^{d}$ arbitrarily and define
\begin{align}
Z \coloneqq P^{\top} (\Phi(X + H) - \Phi(X)) P, \quad K \coloneqq P^{\top}  H P. \label{def:ZandK}
\end{align}
Using~\eqref{eq:X_PDP} and~\eqref{def:ZandK} yields
\begin{align}
Z
&= P^{\top} |X + H|^{3} P - P^{\top} |X|^{3} P \nonumber
\\
&= |P^{\top} (X + H) P|^{3} - |P^{\top} X^{3} P| \nonumber
\\
&= \sqrt{ (D + K)^{6} } - \sqrt{ D^{6} } \nonumber
\\
&= \sqrt{ D^{6} + R + S + T } - \sqrt{D^{6}}, \label{eq:Z_DRST}
\end{align}
where 
\begin{align}
\begin{aligned} \label{def:RST}
R &\coloneqq D^5 K + K D^5 + D^4 K D + D K D^4 + D^3 K D^2 + D^2 K D^3,
\\
S &\coloneqq D^4 K^2 + K^2 D^4 + D^3 K D K + K D K D^3 + D^3 K^2 D + D K^2 D^3 + D^2 K D^2 K + K D^2 K D^2
\\
&\quad \quad  + D^2 K D K D + D K D K D^2 + D K D^3 K + K D^3 K D + D^2 K^2 D^2 + K D^4 K + D K D^2 K D,
\\
T &\coloneqq (D + K)^{6} -D^6 - R - S.
\end{aligned}
\end{align}
From these definitions, we have
\begin{align}
&R =
\begin{bmatrix}
\Lambda^5 K_{\I\I} + K_{\I\I} \Lambda^5 + \Lambda^4 K_{\I\I} \Lambda + \Lambda K_{\I\I} \Lambda^4 + \Lambda^3 K_{\I\I} \Lambda^2 + \Lambda^2 K_{\I\I} \Lambda^3 & \Lambda^5 K_{\I\J}
\\
K_{\I\J}^{\top} \Lambda^5 & O
\\
\end{bmatrix}
, \label{eq:block_R}
\\
&S_{\J\J} = K_{\I\J}^{\top} \Lambda^4 K_{\I\J}. \label{eq:S_JJblock}
\end{align}
Since $\Vert K \Vert_{{\rm F}} = \Vert P^{\top} H P \Vert_{{\rm F}} = \Vert H \Vert_{{\rm F}}$ holds, the second and third definitions of~\eqref{def:RST} imply
\begin{align}
S = o(\Vert H \Vert_{{\rm F}}), \quad T = o(\Vert H \Vert_{{\rm F}}^{2}) \quad (H \to O). \label{eq:ST_bigO}
\end{align}
Now, we denote 
\begin{align}
C \coloneqq |D|^3, \quad W \coloneqq R + S + T. \label{def:CandW}
\end{align}
It then follows from~\eqref{eq:Z_DRST}, \eqref{eq:block_R}, and \eqref{eq:S_JJblock} that
\begin{align}
&Z = \sqrt{C^2 + W} - C, \label{eq:Z}
\\
&W_{\I\I} = \Lambda^5 K_{\I\I} + K_{\I\I} \Lambda^5 + \Lambda^4 K_{\I\I} \Lambda + \Lambda K_{\I\I} \Lambda^4 + \Lambda^3 K_{\I\I} \Lambda^2 + \Lambda^2 K_{\I\I} \Lambda^3 + S_{\I\I} + T_{\I\I}, \label{eq:W_II}
\\
&W_{\I\J} = \Lambda^5 K_{\I\J} + S_{\I\J} + T_{\I\J}, \label{eq:W_IJ}
\\
&W_{\J\J} = K_{\I\J}^{\top} \Lambda^4 K_{\I\J} + T_{\J\J}. \label{eq:W_JJ}
\end{align}
Moreover, using~\cite[Lemma~6.2]{Ts98} implies
\begin{align}
Z_{\I\I} = L_{|\Lambda|^3}^{-1}(W_{\I\I}) + o(\Vert W \Vert_{{\rm F}}) \quad (W \to O),  \label{ineq:Z_II}
\\
Z_{\I\J} = |\Lambda|^{-3} W_{\I\J} + o(\Vert W \Vert_{{\rm F}}) \quad (W \to O). \label{ineq:Z_IJ}
\end{align}
where $L_{|\Lambda|^3} \colon \mathbb{S}^{d-\ell} \to \mathbb{S}^{d-\ell}$ is a linear operator defined by $L_{|\Lambda|^3}(M) \coloneqq |\Lambda|^3 M + M |\Lambda|^3$ for any $M \in \mathbb{S}^{d-\ell}$, and $L^{-1}_{|\Lambda|^3}$ denotes its inverse operator.
On the other hand, we have from~\eqref{eq:Z} that
\begin{align*}
W &= (Z + C)^2 - C^2
\\
&=
\begin{bmatrix}
Z_{\I\I}^2 + Z_{\I\J} Z_{\I\J}^{\top} + Z_{\I\I} |\Lambda|^3 + |\Lambda|^3Z_{\I\I} & Z_{\I\I}Z_{\I\J} + Z_{\I\J}Z_{\J\J} + |\Lambda|^3 Z_{\I\J}
\\
Z_{\I\J}^{\top}Z_{\I\I} + Z_{\J\J} Z_{\I\J}^{\top} + Z_{\I\J}^{\top}|\Lambda|^3  & Z_{\I\J}^{\top}Z_{\I\J} + Z_{\J\J}^2,
\end{bmatrix}
,
\end{align*}
and hence
\begin{align}
W_{\J\J} = Z_{\I\J}^{\top} Z_{\I\J} + Z_{\J\J}^2. \label{eq:W_JJ2}
\end{align}
Now, we recall that $W = R + S + T$ from the second equality of~\eqref{def:CandW}. Then, by noting~\eqref{def:RST}, there exists $b > 0$ such that $\Vert W \Vert_{{\rm F}} \leq b \Vert K \Vert_{{\rm F}} = b \Vert H \Vert_{{\rm F}}$. Thus, it is clear that $H \to O$ implies $W \to O$. According to~\eqref{eq:W_II} and~\eqref{eq:W_IJ}, we obtain $W_{\I\I} = \Lambda^5 K_{\I\I} + K_{\I\I} \Lambda^5 + \Lambda^4 K_{\I\I} \Lambda + \Lambda K_{\I\I} \Lambda^4 + \Lambda^3 K_{\I\I} \Lambda^2 + \Lambda^2 K_{\I\I} \Lambda^3 + S_{\I\I} + T_{\I\I}$ and $W_{\I\J} = \Lambda^5 K_{\I\J} + S_{\I\J} + T_{\I\J}$. It then follows from~\eqref{eq:ST_bigO},~\eqref{ineq:Z_II},~and~\eqref{ineq:Z_IJ} that
\begin{align*}
&\frac{\Vert Z_{\I\I} - L_{|\Lambda|^3}^{-1}(\Lambda^5 K_{\I\I} + K_{\I\I} \Lambda^5 + \Lambda^4 K_{\I\I} \Lambda + \Lambda K_{\I\I} \Lambda^4 + \Lambda^3 K_{\I\I} \Lambda^2 + \Lambda^2 K_{\I\I} \Lambda^3) \Vert_{{\rm F}}}{\Vert H \Vert_{{\rm F}}} 
\\
&\leq \frac{\Vert Z_{\I\I} - L_{|\Lambda|^3}^{-1}(W_{\I\I}) \Vert_{{\rm F}}}{\Vert H \Vert_{{\rm F}}} + \frac{\Vert L_{|\Lambda|^3}^{-1}(S_{\I\I} + T_{\I\I}) \Vert_{{\rm F}}}{\Vert H \Vert_{{\rm F}}} 
\\
&\leq \frac{b \Vert Z_{\I\I} - L_{|\Lambda|^3}^{-1}(W_{\I\I}) \Vert_{{\rm F}}}{\Vert W \Vert_{{\rm F}}} + \frac{\Vert L_{|\Lambda|^3}^{-1} \Vert ( \Vert S \Vert_{{\rm F}} + \Vert T \Vert_{{\rm F}})}{\Vert H \Vert_{{\rm F}}} \to O \quad  (H \to O),
\end{align*}
and
\begin{align*}
\frac{\Vert Z_{\I\J} - |\Lambda|^{-3} \Lambda^5 K_{\I\J} \Vert_{{\rm F}}}{\Vert H \Vert_{{\rm F}}} 
&\leq \frac{\Vert Z_{\I\J} - |\Lambda|^{-3} W_{\I\J} \Vert_{{\rm F}}}{\Vert H \Vert_{{\rm F}}} + \frac{\Vert |\Lambda|^{-3} (S_{\I\J} + T_{\I\J}) \Vert_{{\rm F}}}{\Vert H \Vert_{{\rm F}}}
\\
&\leq \frac{b\Vert Z_{\I\J} - |\Lambda|^{-3} W_{\I\J} \Vert_{{\rm F}}}{\Vert W \Vert_{{\rm F}}} + \frac{\Vert \Lambda^{-3} \Vert_{{\rm F}} ( \Vert S \Vert_{{\rm F}} + \Vert T \Vert_{{\rm F}})}{\Vert H \Vert_{{\rm F}}}  \to O \quad  (H \to O),
\end{align*}
that is
\begin{align*}
Z_{\I\I} - L_{|\Lambda|^3}^{-1}(\Lambda^5 K_{\I\I} + K_{\I\I} \Lambda^5 + \Lambda^4 K_{\I\I} \Lambda + \Lambda K_{\I\I} \Lambda^4 + \Lambda^3 K_{\I\I} \Lambda^2 + \Lambda^2 K_{\I\I} \Lambda^3) = o(\Vert H \Vert_{{\rm F}}) \quad (H \to O),
\\
Z_{\I\J} - |\Lambda| \Lambda K_{\I\J} = o(\Vert H \Vert_{{\rm F}}) \quad (H \to O),
\end{align*}
where note that $|\Lambda|^{-3} \Lambda^5 = |\Lambda| \Lambda$. From these results, there exist matrix-valued functions $U_1 \colon \mathbb{S}^{d} \to \mathbb{S}^{\ell}$ and $U_2 \colon \mathbb{S}^{d} \to \mathbb{R}^{\ell \times (d - \ell)}$ such that
\begin{align}
Z_{\I\I} = L_{|\Lambda|^3}^{-1}(\Lambda^5 K_{\I\I} + K_{\I\I} \Lambda^5 + \Lambda^4 K_{\I\I} \Lambda + \Lambda K_{\I\I} \Lambda^4 + \Lambda^3 K_{\I\I} \Lambda^2 + \Lambda^2 K_{\I\I} \Lambda^3) + \Vert H \Vert_{{\rm F}} U_1(H), \label{eq:Z_II}
\\
Z_{\I\J} = |\Lambda| \Lambda K_{\I\J} + \Vert H \Vert_{{\rm F}} U_2(H),  \label{eq:Z_IJ}
\\
U_1(H) \to O, \quad U_2(H) \to O \quad (H \to O). \label{limit:U1U2}
\end{align}
Now, using~\eqref{eq:Z_IJ} implies
\begin{align}
Z_{\I\J}^{\top} Z_{\I\J} 
&= K_{\I\J}^{\top} \Lambda |\Lambda|^2 \Lambda K_{\I\J} + \Vert H \Vert_{{\rm F}} K_{\I\J}^{\top}\Lambda |\Lambda| U_2(H) + \Vert H \Vert_{{\rm F}} U_2(H)^{\top} |\Lambda| \Lambda K_{\I\J} + \Vert H \Vert_{{\rm F}}^2 U_2(H)^{\top} U_2(H) \nonumber
\\
& = K_{\I\J}^{\top} \Lambda^4 K_{\I\J} + \Vert H \Vert_{{\rm F}} K_{\I\J}^{\top} \Lambda |\Lambda| U_2(H) + \Vert H \Vert_{{\rm F}} U_2(H)^{\top} |\Lambda| \Lambda K_{\I\J} + \Vert H \Vert_{{\rm F}}^2 U_2(H)^{\top} U_2(H). \label{eq:Z_IJ_U}
\end{align}
By exploiting the second equality of~\eqref{eq:ST_bigO}, there exists a matrix-valued function $U_3 \colon \mathbb{S}^{d} \to \mathbb{S}^{d}$ such that
\begin{align}
T = \Vert H \Vert_{{\rm F}}^2 U_3(H), \quad U_3(H) \to O ~ (H \to O). \label{eq:T_o}
\end{align}
Meanwhile, equalities~\eqref{eq:W_JJ} and~\eqref{eq:W_JJ2} lead to $Z_{\J\J}^2 = W_{\J\J} - Z_{\I\J}^{\top} Z_{\I\J} = K_{\I\J}^{\top} \Lambda^4 K_{\I\J} +T_{\J\J} - Z_{\I\J}^{\top} Z_{\I\J}$. This fact,~\eqref{eq:Z_IJ_U}, and \eqref{eq:T_o} derive
\begin{align}
\frac{ \Vert Z_{\J\J}^2 \Vert_{{\rm F}} }{ \Vert H \Vert_{{\rm F}}^2 } 
&= \frac{ \Vert T_{\J\J} - \Vert H \Vert_{{\rm F}} K_{\I\J}^{\top} \Lambda |\Lambda| U_2(H) - \Vert H \Vert_{{\rm F}} U_2(H)^{\top} |\Lambda| \Lambda K_{\I\J} - \Vert H \Vert_{{\rm F}}^2 U_2(H)^{\top} U_2(H) \Vert_{{\rm F}} }{\Vert H \Vert_{{\rm F}}^2} \nonumber
\\
&\leq \frac{ \Vert T \Vert_{{\rm F}} + 2\Vert H \Vert_{{\rm F}}^2 \Vert \Lambda \Vert_{{\rm F}}^2 \Vert U_2(H) \Vert_{{\rm F}} + \Vert H \Vert_{{\rm F}}^2 \Vert U_2(H) \Vert_{{\rm F}}^2 }{\Vert H \Vert_{{\rm F}}^2} \nonumber
\\
&\leq \Vert U_3(H) \Vert_{{\rm F}} + 2 \Vert \Lambda \Vert_{{\rm F}}^2 \Vert U_2(H) \Vert_{{\rm F}} + \Vert U_2(H) \Vert_{{\rm F}}^2, \label{ineq:Z_JJ_2}
\end{align}
where note that $\Vert T_{\J\J} \Vert_{{\rm F}} \leq \Vert T \Vert_{{\rm F}}$ and $\Vert K_{\I\J} \Vert_{{\rm F}} \leq \Vert K \Vert_{{\rm F}} = \Vert H \Vert_{{\rm F}}$. Combining~\eqref{limit:U1U2},~\eqref{eq:T_o}, and~\eqref{ineq:Z_JJ_2} yields $Z_{\J\J}^2 = o(\Vert H \Vert_{{\rm F}}^2)~(H \to O)$. Namely, there exists a matrix-valued function $U_4 \colon \mathbb{S}^{d} \to \mathbb{S}^{(d-\ell)}$ such that 
\begin{align}
Z_{\J\J}^2 = \Vert H \Vert_{{\rm F}}^2 U_4(H), \quad U_4(H) \succeq O, \quad U_4(H) \to O~(H \to O). \label{eq:Z_JJ_U4}
\end{align}
Since $Z + C = \sqrt{C^2 + W}$ from~\eqref{eq:Z}, we have
\begin{align}
\begin{bmatrix}
Z_{\I\I} + |\Lambda|^3 & Z_{\I\J}
\\
Z_{\I\J}^{\top} & Z_{\J\J}
\end{bmatrix}
= \sqrt{C^2 + W} \succeq O \quad \Longrightarrow \quad  Z_{\J\J} \succeq O. \label{eq:Z_JJ_psd}
\end{align}
We have from~\eqref{eq:Z_JJ_U4} and~\eqref{eq:Z_JJ_psd} that
\begin{align}
Z_{\J\J} = \Vert H \Vert_{{\rm F}} \sqrt{U_{4}(H)}. \label{eq:Z_JJ}
\end{align}
By using~\eqref{eq:Z_II}, \eqref{eq:Z_IJ}, and~\eqref{eq:Z_JJ} the matrix $Z$ can be represented as follows:
\begin{align}
Z &=
\begin{bmatrix}
Z_{\I\I} & Z_{\I\J}
\\
Z_{\I\J}^{\top} & Z_{\J\J}
\end{bmatrix}
\nonumber
\\
&=
\begin{bmatrix}
L_{|\Lambda|^3}^{-1}(\Lambda^5 K_{\I\I} + K_{\I\I} \Lambda^5 + \Lambda^4 K_{\I\I} \Lambda + \Lambda K_{\I\I} \Lambda^4 + \Lambda^3 K_{\I\I} \Lambda^2 + \Lambda^2 K_{\I\I} \Lambda^3) & |\Lambda| \Lambda K_{\I\J}
\\
K_{\I\J}^{\top} \Lambda |\Lambda| & O
\end{bmatrix}
+ \Vert H \Vert_{{\rm F}} U(H), \label{eq:Z_submat}
\end{align}
where
\begin{align} \label{def:matU}
U(H) \coloneqq
\begin{bmatrix}
U_1(H) & U_2(H)
\\
U_2(H)^{\top} & \sqrt{U_4(H)}
\end{bmatrix}
. 
\end{align}
Now, we define a linear operator ${\cal L}_{X} \colon \mathbb{S}^{d} \to \mathbb{S}^{d}$ as
\begin{align} \label{def:operL}
\begin{aligned}
{\cal L}_{X}(H) &\coloneqq P
\begin{bmatrix}
{\cal L}_{\I\I}(H) & {\cal L}_{\I\J}(H)
\\
{\cal L}_{\I\J}(H)^{\top} & O
\end{bmatrix}
P^{\top},
\\
{\cal L}_{\I\I}(H) &\coloneqq L_{|\Lambda|^3}^{-1}(\Lambda^5 P_{\I}^{\top} H P_{\I} + P_{\I}^{\top} H P_{\I} \Lambda^5 + \Lambda^4 P_{\I}^{\top} H P_{\I} \Lambda + \Lambda P_{\I}^{\top} H P_{\I} \Lambda^4 + \Lambda^3 P_{\I}^{\top} H P_{\I} \Lambda^2 + \Lambda^2 P_{\I}^{\top} H P_{\I} \Lambda^3),
\\
{\cal L}_{\I\J}(H) &\coloneqq  |\Lambda| \Lambda P_{\I}^{\top} H P_{\J}.
\end{aligned}
\end{align}
Recall that $P = [P_{\I} ~ P_{\J}]$ from~\eqref{eq:X_PDP}. Then, noting the second equality of~\eqref{def:ZandK} implies $K_{\I\I} = P_{\I}^{\top} H P_{\I}$, $K_{\I\J} = P_{\I}^{\top} H P_{\J}$, and $K_{\J\J} = P_{\J}^{\top} H P_{\J}$. It then follows from~\eqref{eq:Z_submat} and~\eqref{def:operL} that $Z = P^{\top} {\cal L}_{X}(H) P + \Vert H \Vert_{{\rm F}} U(H)$. This equality and the first equality of~\eqref{def:ZandK} derive
\begin{align}
\Phi(X + H) - \Phi(X) = PZP^{\top} = {\cal L}_{X}(H) + \Vert H \Vert_{{\rm F}} PU(H)P^{\top}. \label{eq:PhiXHX}
\end{align}
On the other hand, we obtain $U(H) \to O~(H \to O)$ by~\eqref{limit:U1U2}, \eqref{eq:Z_JJ_U4}, and~\eqref{def:matU}. Then, exploiting~\eqref{eq:PhiXHX} leads to $\Phi(X+H) = \Phi(X) + {\cal L}_X(H) + o(\Vert H \Vert_{{\rm F}}) ~ (H \to O)$. From the above proof, notice that this result is also true in the case where $X$ is nonsingular, that is, $P = P_{\I}$ and $D = \Lambda$ in~\eqref{eq:X_PDP}. Consequently, the matrix-valued function $\Phi$ is differentiable, and $\D \Phi(X) H = {\cal L}_{X} (H)$ holds.
\par
Since the function ${\cal Q}(X) = [X]_{+}^3$ is represented as ${\cal Q}(X) = \frac{1}{2}(X^3 + \Phi(X))$, it is differentiable and its derivative is given by
\begin{align}
{\rm D} {\cal Q}(X) H = \frac{1}{2}(X^2 H + H X^2 + XHX + {\cal L}_{X}(H)). \label{eq:DiffPhi}
\end{align}
Now, we denote by $\mu_{j} \in \R$ the $(j, j)$-entry of the diagonal matrix $\Lambda \in \S^{d}$. Since $| \Lambda | \succ O$ is satisfied from~\eqref{eq:X_PDP}, it is clear that $\mu_{j} \not = 0$ for all $j \in \{ 1, \ldots, \ell \}$. By using $\mu_{1}, \ldots, \mu_{\ell}$, the inverse operator $L_{| \Lambda |^{3}}^{-1} \colon \S^{\ell} \to \S^{\ell}$ can be calculated as follows: $L_{|\Lambda|^{3}}^{-1}(M) = [ M_{ij}(|\mu_{i}|^{3} + |\mu_{j}|^{3})^{-1} ]_{ij}$ for any $M \in \S^{d}$. Then, combining~\eqref{eq:X_PDP},~\eqref{def:operL}, and~\eqref{eq:DiffPhi} implies the desired results. 
\end{proof}

\noindent
Lemma~\ref{lem:derivative_Q} shows the differentiability of $\mathcal{Q}$, but does not ensure that it is continuously differentiable. Hence, we also show the continuity of $\D \mathcal{Q}$.

\begin{lemma} \label{leem:continuity_DQ}
The derivative $\D \mathcal{Q}$ is continuous on $\S^{d}$.
\end{lemma}

\begin{proof}
Let $X \in \S^{d}$ be arbitrary, and let $\{ X_{k} \} \subset \S^{d}$ be an arbitrary sequence satisfying $X_{k} \to X~(k \to \infty)$. Now, rcall that $\lambda_{1}(X) \geq \cdots \geq \lambda_{d}(X)$ and $\lambda_{1}(X_{k}) \geq \cdots \geq \lambda_{d}(X_{k})$ for any $k \in \mathbb{N}$. We define $\A \coloneqq \{ j \in \mathbb{N}; \lambda_{j}(X) > 0 \}$, $\B \coloneqq \{ j \in \mathbb{N}; \lambda_{j}(X) = 0 \}$, and $\C \coloneqq \{ j \in \mathbb{N}; \lambda_{j}(X) < 0 \}$. By exploiting the continuity of eigenvalues, there exists $\bar{m} \in \mathbb{N}$ such that 
\begin{align}
\forall k \geq \bar{m}, \quad \forall (i, j) \in \A \times \C, \quad \lambda_{i}(X_{k}) > 0, \quad \lambda_{j}(X_{k}) < 0. \label{eig:Xk_pn}
\end{align}
From now on, we suppose that $k \geq \bar{m}$. We also assume that $X$ and $X_{k}$ have the following matrix decompositions:
\begin{align} \label{eq:XXkPPk}
\begin{gathered}
X = P(X) {\rm diag}(\lambda_{1}(X), \ldots, \lambda_{d}(X)) P(X)^{\top}, \quad P(X) P(X)^{\top} = I, \quad P(X) = \big[ p_{1}(X) \, \cdots \, p_{d}(X) \big],
\\
X_{k} = Q(X_{k}) {\rm diag}(\lambda_{1}(X_{k}), \ldots, \lambda_{d}(X_{k})) Q(X_{k})^{\top}, \quad Q(X_{k}) Q(X_{k})^{\top} = I, \quad Q(X_{k}) = \big[ q_{1}(X_{k}) \, \cdots \, q_{d}(X_{k}) \big],
\end{gathered}
\end{align}
where $p_{j}(X) \in \R^{d}$ and $q_{j}(X_{k}) \in \R^{d}$ for each $j \in \{1, \ldots, d \}$. Let us define $\A_{k} \coloneqq \{ j \in \mathbb{N}; \lambda_{j}(X_{k}) > 0 \}$, $\B_{k} \coloneqq \{ j \in \mathbb{N}; \lambda_{j}(X_{k}) = 0 \}$, and $\C_{k} \coloneqq \{ j \in \mathbb{N}; \lambda_{j}(X_{k}) < 0 \}$. Then, noting~\eqref{eig:Xk_pn} implies $\A \subset \A_{k}$ and $\C \subset \C_{k}$. Hence, we define $\B_{k}^{+} \coloneqq \A_{k} \backslash \A$ and $\B_{k}^{-} \coloneqq \C_{k} \backslash \C$. These definitions derive
\begin{align}
\A_{k} = \A \cup \B_{k}^{+}, \quad \C_{k} = \C \cup \B_{k}^{-}, \quad \B = \B_{k}^{+} \cup \B_{k} \cup \B_{k}^{-}. \label{set:ABCk}
\end{align}
Let $H \in \S^{d}$ be arbitrary. Using~\eqref{eq:XXkPPk} and Lemma~\ref{lem:derivative_Q} yields
\begin{gather}
\begin{gathered} \label{eq:DQXDQXk}
\D \mathcal{Q}(X) H = P(X) ( C(X) \circ P(X)^{\top} H P(X) ) P(X)^{\top}, 
\\
\D \mathcal{Q}(X_{k}) H = Q(X_{k}) ( C(X_{k}) \circ Q(X_{k})^{\top} H Q(X_{k}) ) Q(X_{k})^{\top},
\end{gathered}
\end{gather}
where $C(X) \in \S^{d}$ and $C(X_{k}) \in \S^{d}$ are defined by
\begin{align}
[C(X)]_{ij} & =
\begin{cases} \label{def:matCX}
\lambda_{i}(X)^{2} + \lambda_{i}(X) \lambda_{j}(X) + \lambda_{j}(X)^{2} & {\rm if} ~ i \in \A, ~ j \in \A,
\\
\lambda_{i}(X)^{2} & {\rm if} ~ i \in \A, ~ j \in \B,
\\
\lambda_{i}(X)^{3} ( \lambda_{i}(X) - \lambda_{j}(X) )^{-1} & {\rm if} ~ i \in \A, ~ j \in \C,
\\
\lambda_{j}(X)^{2} & {\rm if} ~ i \in \B, ~ j \in \A,
\\
\lambda_{j}(X)^{3} (\lambda_{j}(X) - \lambda_{i}(X))^{-1} & {\rm if} ~ i \in \C, ~ j \in \A,
\\
0 & {\rm otherwise},
\end{cases}
\\
[C(X_{k})]_{ij} & =
\begin{cases} \label{def:matCXk}
\lambda_{i}(X_{k})^{2} + \lambda_{i}(X_{k}) \lambda_{j}(X_{k}) + \lambda_{j}(X_{k})^{2} & {\rm if} ~ i \in \A_{k}, ~ j \in \A_{k},
\\
\lambda_{i}(X_{k})^{2} & {\rm if} ~ i \in \A_{k}, ~ j \in \B_{k},
\\
\lambda_{i}(X_{k})^{3} ( \lambda_{i}(X_{k}) - \lambda_{j}(X_{k}) )^{-1} & {\rm if} ~ i \in \A_{k}, ~ j \in \C_{k},
\\
\lambda_{j}(X_{k})^{2} & {\rm if} ~ i \in \B_{k}, ~ j \in \A_{k},
\\
\lambda_{j}(X_{k})^{3} (\lambda_{j}(X_{k}) - \lambda_{i}(X_{k}))^{-1} & {\rm if} ~ i \in \C_{k}, ~ j \in \A_{k},
\\
0 & {\rm otherwise},
\end{cases}
\end{align}
respectively.
\par
In the following, we prove that $C(X_{k}) \to C(X)$ as $k \to \infty$. Recall that the last equality of~\eqref{set:ABCk} is satisfied, and hence we obtain
\begin{align}
\forall j \in \B_{k}^{+} \cup \B_{k} \cup \B_{k}^{-} = \B, \quad \lim_{k \to \infty} \lambda_{j}(X_{k}) = \lambda_{j}(X) = 0. \label{eq:B_lambda}
\end{align}
There are three possible cases: (i) $(i, j) \in \A_{k} \times \A_{k}$; (ii) $(i, j) \in (\A_{k} \times \B_{k}) \cup (\B_{k} \times \A_{k})$; (iii) $(i, j) \in (\A_{k} \times \C_{k}) \cup (\C_{k} \times \A_{k})$. By noting~\eqref{set:ABCk},~\eqref{def:matCX},~\eqref{def:matCXk}, and $\lambda_{j}(X_{k}) \to \lambda_{j}(X)~(k \to \infty)$ for any $j \in \{1, \ldots, d\}$, we discuss these cases:
\begin{description}
\item[Case~(i)] If $(i, j) \in \A \times \A$, it is clear that $[C(X_{k})]_{ij} \to [C(X)]_{ij}~(k \to \infty)$. If $(i, j) \in \A \times \B_{k}^{+}$ or $(i, j) \in \B_{k}^{+} \times \A$, then $(i, j) \in \A \times \B$ or $(i, j) \in \B \times \A$ holds, and using~\eqref{eq:B_lambda} derives $[C(X_{k})]_{ij} \to [C(X)]_{ij}~(k \to \infty)$. Finally, if $(i, j) \in \B_{k}^{+} \times \B_{k}^{+}$, then $(i, j) \in \B \times \B$, and thus again using~\eqref{eq:B_lambda} leads to $[C(X_{k})]_{ij} \to [C(X)]_{ij}~(k \to \infty)$.

\item[Case~(ii)] For this case, we readily have $[C(X_{k})]_{ij} \to [C(X)]_{ij}~(k \to \infty)$.

\item[Case~(iii)] To begin with, we consider the case where $(i, j) \in \A_{k} \times \C_{k}$. For this case, there occur the following four cases: $(i, j) \in \A \times \C$, $(i, j) \in \A \times \B_{k}^{-}$, $(i, j) \in \B_{k}^{+} \times \C$, and $(i, j) \in \B_{k}^{+} \times \B_{k}^{-}$. The first case directly implies $[C(X_{k})]_{ij} \to [C(X)]_{ij}~(k \to \infty)$. For the second case, we see $\lambda_{i}(X_{k}) > 0$, $\lambda_{j}(X_{k}) < 0$, and $(i, j) \in \A \times \B$. Hence, we have from~\eqref{eq:B_lambda} that
\begin{align*}
[C(X_{k})]_{ij} = \frac{\lambda_{i}(X_{k})^{3}}{\lambda_{i}(X_{k}) - \lambda_{j}(X_{k})} \to \frac{\lambda_{i}(X)^{3}}{\lambda_{i}(X)} = \lambda_{i}(X)^{2} = [C(X)]_{ij} ~~ (k \to \infty).
\end{align*}
The third and fourth cases show that $\lambda_{i}(X_{k}) > 0$, $\lambda_{j}(X_{k}) < 0$, and $[C(X)]_{ij} = 0$. It then follows from $\lambda_{i}(X_{k}) < \lambda_{i}(X_{k}) - \lambda_{j}(X_{k})$ that
\begin{align*}
0 < \frac{\lambda_{i}(X_{k})^{3}}{\lambda_{i}(X_{k}) - \lambda_{j}(X_{k})} < \frac{\lambda_{i}(X_{k})^{3}}{\lambda_{i}(X_{k})} = \lambda_{i}(X_{k})^{2} \to 0 ~~ (k \to \infty),
\end{align*}
that is, $[C(X_{k})]_{ij} = \lambda_{i}(X_{k})^{3} ( \lambda_{i}(X_{k}) - \lambda_{j}(X_{k}) )^{-1} \to 0 = [C(X)]_{ij} ~ (k \to \infty)$. The case of $(i, j) \in \C_{k} \times \A_{k}$ can be shown in a way similar to the case of $(i, j) \in \A_{k} \times \C_{k}$. 
\end{description}
Note that the other cases imply $[C(X)]_{ij} = [C(X_{k})]_{ij} = 0$. Therefore, these arguments ensure $C(X_{k}) \to C(X) ~ (k \to \infty)$. 
\par
Next, we show that 
\begin{align}
\lim_{k \to \infty} \Vert \D \mathcal{Q} (X_{k}) H - \D \mathcal{Q} (X) H \Vert_{{\rm F}} = 0. \label{lim:DQXkH}
\end{align}
Let $\mu_{s}(X)$ be the $s$-th largest eigenvalue of $X$. Suppose that the smallest eigenvalue of $X$ is $\mu_{\ell}(X)$, that is, $\mu_{1}(X) > \cdots > \mu_{\ell}(X)$. We define $\I_{s} \coloneqq \{ i \in \mathbb{N}; \, \lambda_{i}(X) = \mu_{s}(X) \}$ for each $s \in \{ 1, \ldots, \ell \}$. Note that $\I_{i} \cap \I_{j} = \emptyset$ for $i \not = j$ because $\mu_{1}(X) > \cdots > \mu_{\ell}(X)$. Let $Y \in \S^{d}$ be an arbitrary matrix in a neighborhood of $X$, and let $\Gamma_{s}(Y)$ be the projection matrix onto the eigenspace corresponding to the eigenvalues $\lambda_{i}(Y)$ for all $i \in \I_{s}$. It is known that each $\Gamma_{s}$ is an analytic function in a sufficiently small neighborhood of $X$. Hence, we assume that all the matrix-valued functions $\Gamma_{1}, \ldots, \Gamma_{\ell}$ are analytic in $B(X, \delta)$, where $\delta > 0$ is some constant. Since $\{ X_{k} \}$ converges to $X$, there exists $\bar{n} \geq \bar{m}$ such that $X_{k} \in B(X, \delta)$ for each $k \geq \bar{n}$, where $\bar{m}$ is the constant defined in~\eqref{eig:Xk_pn}. From now on, we suppose $k \geq \bar{n}$. Recall that $P(X) = \big[ p_{1}(X) \, \cdots \, p_{d}(X) \big]$ and $Q(X_{k}) = \big[ q_{1}(X_{k}) \, \cdots \, q_{d}(X_{k}) \big]$ from~\eqref{eq:XXkPPk}. It then follows from~\eqref{eq:DQXDQXk} that
\begin{align}
\D \mathcal{Q}(X_{k}) H - \D \mathcal{Q}(X) H 
&= \sum_{i=1}^{d} \sum_{j=1}^{d} \Big( [C(X_{k})]_{ij} - [C(X)]_{ij} \Big)  q_{i}(X_{k}) q_{i}(X_{k})^{\top} H q_{j}(X_{k}) q_{j}(X_{k})^{\top} \nonumber
\\
& \qquad + \sum_{i=1}^{d} \sum_{j=1}^{d} [C(X)]_{ij} q_{i}(X_{k}) q_{i}(X_{k})^{\top} H \Big( q_{j}(X_{k}) q_{j}(X_{k})^{\top} - p_{j}(X) p_{j}(X)^{\top} \Big) \nonumber
\\
& \qquad \qquad + \sum_{i=1}^{d} \sum_{j=1}^{d} [C(X)]_{ij} \Big( q_{i}(X_{k}) q_{i}(X_{k})^{\top} - p_{i}(X) p_{i}(X)^{\top} \Big) H p_{j}(X) p_{j}(X)^{\top}. \label{ineq:DQH}
\end{align}
Now, we evaluate each term of the right-hand side of the above. The first term can be evaluated as follows:
\begin{align}
\left\Vert \sum_{i=1}^{d} \sum_{j=1}^{d} \Big( [C(X_{k})]_{ij} - [C(X)]_{ij} \Big)  q_{i}(X_{k}) q_{i}(X_{k})^{\top} H q_{j}(X_{k}) q_{j}(X_{k})^{\top} \right\Vert_{{\rm F}}
 \leq d \Vert H \Vert _{{\rm F}} \Vert C(X_{k}) - C(X) \Vert_{{\rm F}}, \label{ineq:term1}
\end{align}
where note that $\Vert q_{i}(X_{k}) q_{i}(X_{k})^{\top} \Vert_{{\rm F}} = 1$. Next, we evaluate the second term of~\eqref{ineq:DQH}. For simplicity, let us define $V_{ij} \coloneqq [C(X)]_{ij} q_{i}(X_{k}) q_{i}(X_{k})^{\top} H ( q_{j}(X_{k}) q_{j}(X_{k})^{\top} - p_{j}(X) p_{j}(X)^{\top} )$. For any $\alpha, \, \beta \in \{ 1, \ldots, \ell \}$, the definition of $C(X)$ implies that if $i \in \I_{\alpha}$ and $j \in \I_{\beta}$, the value of $[C(X)]_{ij}$ is constant although it depends on $\alpha$ and $\beta$. Hence, we have
\begin{align}
\left\Vert \sum_{i \in \I_{\alpha}} \sum_{j \in \I_{\beta}} V_{ij} \right\Vert_{{\rm F}} \leq \Vert C(X) \Vert_{{\rm F}} \Vert H \Vert_{{\rm F}} \left\Vert \sum_{i \in \I_{\alpha}} q_{i}(X_{k}) q_{i}(X_{k})^{\top} \right\Vert_{{\rm F}} \left\Vert \sum_{j \in \I_{\beta}} \Big( q_{j}(X_{k}) q_{j}(X_{k})^{\top} - p_{j}(X) p_{j}(X)^{\top} \Big) \right\Vert_{{\rm F}}. \label{ineq:constant_Vij}
\end{align}
Meanwhile, note that $\sum_{j \in \I_{s}} q_{j}(X_{k}) q_{j}(X_{k})^{\top} = \Gamma_{s}(X_{k})$ and $\sum_{j \in \I_{s}} p_{j}(X) p_{j}(X)^{\top} = \Gamma_{s}(X)$ for each $s \in \{ 1, \ldots, \ell \}$. It then follows from~\eqref{ineq:constant_Vij} that
\begin{align}
& \left\Vert \sum_{i=1}^{d} \sum_{j=1}^{d} [C(X)]_{ij} q_{i}(X_{k}) q_{i}(X_{k})^{\top} H \Big( q_{j}(X_{k}) q_{j}(X_{k})^{\top} - p_{j}(X) p_{j}(X)^{\top} \Big) \right\Vert_{{\rm F}} \nonumber
\\
& \leq \left\Vert \sum_{i \in \I_{1}} \sum_{j \in \I_{1}} V_{ij} \right\Vert_{{\rm F}} + \cdots + \left\Vert \sum_{i \in \I_{1}} \sum_{j \in \I_{\ell}} V_{ij} \right\Vert_{{\rm F}}
 \nonumber
\\
& \qquad + \left\Vert \sum_{i \in \I_{2}} \sum_{j \in \I_{1}} V_{ij} \right\Vert_{{\rm F}} + \cdots + \left\Vert \sum_{i \in \I_{2}} \sum_{j \in \I_{\ell}} V_{ij} \right\Vert_{{\rm F}} + \cdots + \left\Vert \sum_{i \in \I_{\ell}} \sum_{j \in \I_{1}} V_{ij} \right\Vert_{{\rm F}} + \cdots + \left\Vert \sum_{i \in \I_{\ell}} \sum_{j \in \I_{\ell}} V_{ij} \right\Vert_{{\rm F}} \nonumber
\\
& \leq \Vert C(X) \Vert_{{\rm F}} \Vert H \Vert_{{\rm F}} | \I_{1} | \Big( \left\Vert \Gamma_{1}(X_{k}) - \Gamma_{1}(X) \right\Vert_{{\rm F}} + \cdots +  \left\Vert \Gamma_{\ell}(X_{k}) - \Gamma_{\ell}(X) \right\Vert_{{\rm F}} \Big) \nonumber
\\
& \qquad + \Vert C(X) \Vert_{{\rm F}} \Vert H \Vert_{{\rm F}} | \I_{2} | \Big( \left\Vert \Gamma_{1}(X_{k}) - \Gamma_{1}(X) \right\Vert_{{\rm F}} + \cdots +  \left\Vert \Gamma_{\ell}(X_{k}) - \Gamma_{\ell}(X) \right\Vert_{{\rm F}} \Big) + \cdots \nonumber
\\
& \qquad \qquad + \Vert C(X) \Vert_{{\rm F}} \Vert H \Vert_{{\rm F}} | \I_{\ell} | \Big( \left\Vert \Gamma_{1}(X_{k}) - \Gamma_{1}(X) \right\Vert_{{\rm F}} + \cdots +  \left\Vert \Gamma_{\ell}(X_{k}) - \Gamma_{\ell}(X) \right\Vert_{{\rm F}} \Big). \label{ineq:term2}
\end{align}
Using a way similar to the above evaluation, we also obtain
\begin{align}
& \left\Vert \sum_{i=1}^{d} \sum_{j=1}^{d} [C(X)]_{ij} \Big( q_{i}(X_{k}) q_{i}(X_{k})^{\top} - p_{i}(X) p_{i}(X)^{\top} \Big) H p_{j}(X) p_{j}(X)^{\top} \right\Vert_{{\rm F}} \nonumber
\\
& \leq \Vert C(X) \Vert_{{\rm F}} \Vert H \Vert_{{\rm F}} | \I_{1} | \Big( \left\Vert \Gamma_{1}(X_{k}) - \Gamma_{1}(X) \right\Vert_{{\rm F}} + \cdots +  \left\Vert \Gamma_{\ell}(X_{k}) - \Gamma_{\ell}(X) \right\Vert_{{\rm F}} \Big) \nonumber
\\
& \qquad + \Vert C(X) \Vert_{{\rm F}} \Vert H \Vert_{{\rm F}} | \I_{2} | \Big( \left\Vert \Gamma_{1}(X_{k}) - \Gamma_{1}(X) \right\Vert_{{\rm F}} + \cdots +  \left\Vert \Gamma_{\ell}(X_{k}) - \Gamma_{\ell}(X) \right\Vert_{{\rm F}} \Big) + \cdots \nonumber
\\
& \qquad \qquad + \Vert C(X) \Vert_{{\rm F}} \Vert H \Vert_{{\rm F}} | \I_{\ell} | \Big( \left\Vert \Gamma_{1}(X_{k}) - \Gamma_{1}(X) \right\Vert_{{\rm F}} + \cdots +  \left\Vert \Gamma_{\ell}(X_{k}) - \Gamma_{\ell}(X) \right\Vert_{{\rm F}} \Big), \label{ineq:term3}
\end{align}
where note that $\Vert p_{j}(X) p_{j}(X)^{\top} \Vert_{{\rm F}} = 1$. We have from~\eqref{ineq:DQH}--\eqref{ineq:term3} that~\eqref{lim:DQXkH} is satisfied.
\par
Recall that $\Vert \D \mathcal{Q} (X_{k}) - \D \mathcal{Q} (X) \Vert = \sup \{ \Vert \D \mathcal{Q} (X_{k}) U - \D \mathcal{Q} (X) U \Vert_{{\rm F}}; \Vert U \Vert_{{\rm F}} \leq 1 \}$. From the definition of the operator norm, there exists a sequence $\{ U_{i} \} \subset \S^{d}$ such that
\begin{align}
\Vert U_{i} \Vert_{{\rm F}} \leq 1, \quad \Vert \D \mathcal{Q} (X_{k}) - \D \mathcal{Q} (X) \Vert < \frac{1}{i} + \Vert \D \mathcal{Q} (X_{k}) U_{i} - \D \mathcal{Q} (X) U_{i} \Vert_{{\rm F}} \quad  \forall i \in \mathbb{N}. \label{ineq:opernorm}
\end{align}
The first inequality of~\eqref{ineq:opernorm} implies that $\{ U_{i} \} \subset \S^{d}$ has an accumulation point $U \in \S^{d}$, that is, there exists a subsequence $\{ U_{i_{j}} \} \subset \{ U_{i} \}$ such that $U_{i_{j}} \to U$ as $j \to \infty$.
Since the second inequality of~\eqref{ineq:opernorm} guarantees $\Vert \D \mathcal{Q} (X_{k}) - \D \mathcal{Q} (X) \Vert < \frac{1}{i_{j}} + \Vert \D \mathcal{Q} (X_{k}) U_{i_{j}} - \D \mathcal{Q} (X) U_{i_{j}} \Vert_{{\rm F}}$ for all $j \in \mathbb{N}$, taking $j \to \infty$ derives
\begin{align*}
\Vert \D \mathcal{Q} (X_{k}) - \D \mathcal{Q} (X) \Vert \leq \Vert \D \mathcal{Q} (X_{k}) U - \D \mathcal{Q} (X) U \Vert_{{\rm F}}.
\end{align*}
It then follows from equality~\eqref{lim:DQXkH} with $H=U$ that $\D \mathcal{Q}(X_{k}) \to \D \mathcal{Q}(X)$ as $k \to \infty$. Therefore, the derivative $\D \mathcal{Q}$ is continuous on $\S^{d}$.
\end{proof}

\noindent
Exploiting Lemmas~\ref{lem:derivative_thirdterm},~\ref{lem:derivative_Q}, and~\ref{leem:continuity_DQ}, we can provide the gradient and the Hessian of the proposed function. To represent the Hessian, for any $x \in \R^{n}$, we use the notation below:
\begin{align*}
G_{ij}(x) \coloneqq \frac{\partial}{\partial x_{i}} \left( \frac{\partial}{\partial x_{j}} G(x) \right) \quad (\, i = 1, \ldots, n, ~ j = 1, \ldots, n \,).
\end{align*}

\begin{theorem} \label{thm:F_derivative_Hessian}
For $v \in \R^{m}$, $M \in \S^{d}$, $\rho > 0$, $\sigma > 0$, and $\tau > 0$, the function $F(\, \cdot \, ; v, M, \rho, \sigma, \tau) \colon \R^{n} \to \R$ is twice continuously differentiable on $\R^{n}$, and its gradient and Hessian are given as follows:
\begin{align*}
\nabla F(x; v, M, \rho, \sigma, \tau) & \textstyle = \rho \nabla f(x) - \sigma \tau \nabla g(x) ( \frac{1}{\tau} y - g(x) ) - \sigma \tau \D G(x)^{\ast} [ \frac{1}{\tau} M - G(x) ]_{+}^{3},
\\
\nabla^{2} F(x; v, M, \rho, \sigma, \tau) & = \rho \nabla^{2} f(x) - \sigma \tau \sum_{j=1}^{m} ( \textstyle \frac{1}{\tau} v_{j} - g_{j}(x) ) \nabla^{2} g_{j}(x) + \sigma \tau \nabla g(x) \nabla g(x)^{\top}
\\
& \qquad - \sigma \tau
\begin{bmatrix}
\langle G_{11}(x), [ \frac{1}{\tau} M - G(x)]_{+}^{3} \rangle & \cdots & \langle G_{1n}(x), [ \frac{1}{\tau} M - G(x)]_{+}^{3} \rangle
\\
\vdots & \ddots & \vdots
\\
\langle G_{n1}(x), [ \frac{1}{\tau} M - G(x)]_{+}^{3} \rangle & \cdots & \langle G_{nn}(x), [ \frac{1}{\tau} M - G(x)]_{+}^{3} \rangle
\end{bmatrix}
\\
& \qquad + \sigma \tau
\begin{bmatrix}
\langle G_{1}(x), \D \mathcal{Q}(\frac{1}{\tau} M - G(x))G_{1}(x) \rangle & \cdots & \langle G_{1}(x), \D \mathcal{Q}( \frac{1}{\tau} M - G(x))G_{n}(x) \rangle
\\
\vdots & \ddots & \vdots
\\
\langle G_{n}(x), \D \mathcal{Q}( \frac{1}{\tau} M - G(x))G_{1}(x) \rangle & \cdots & \langle G_{n}(x), \D \mathcal{Q}( \frac{1}{\tau} M - G(x))G_{n}(x) \rangle
\end{bmatrix}
\end{align*}
for any $x \in \R^{n}$, where $\mathcal{Q}$ is defined by~\eqref{def:funcQ}.
\end{theorem}

\begin{proof}
By combining Lemma~\ref{lem:derivative_thirdterm} and the chain rule of differentiation, we obtain the gradient $\nabla F(x; v, M, \rho, \sigma, \tau)$. The Hessian $\nabla^{2} F(x; v, M, \rho, \sigma, \tau)$ is derived from Lemma~\ref{lem:derivative_Q} and the chain rule of differentiation, and its continuity is ensured by the form of the Hessian and Lemma~\ref{leem:continuity_DQ}. Therefore, we conclude that $F(\, \cdot \, ; v, M, \rho, \sigma, \tau) \colon \R^{n} \to \R$ is twice continuously differentiable on $\R^{n}$.
\end{proof}

\section{A penalty method for second-order stationary points} \label{sec:methods}
This section presents a penalty method to find second-order stationary points for problem~\eqref{NSDP} by exploiting the penalty function proposed in the previous section. 
To begin with, we define $\F(\, \cdot \,; \gamma) \colon \R^{n} \to \R$ and $\P \colon \R^{n} \to \R$ as
\begin{align*}
& \F(x; \gamma) \coloneqq F(x; 0, O, 1, \gamma, 1) = f(x) + \frac{\gamma}{2} \Vert g(x) \Vert^{2} + \frac{\gamma}{4} {\rm tr}([-G(x)]_{+}^{4}),
\\
& \P(x) \coloneqq F(x; 0, O, 0, 1, 1) = \frac{1}{2} \Vert g(x) \Vert^{2} + \frac{1}{4} {\rm tr} ([-G(x)]_{+}^{4}),
\end{align*}
respectively, where $\gamma > 0$ is a parameter. For the subsequent arguments, the generalized Polyak-\L ojasiewicz inequality is provided as follows.
\begin{definition}
We say that $\bar{x} \in \R^{n}$ satisfies the generalized Polyak-\L ojasiewicz inequality if there exist $\nu > 0$ and $\Psi \colon B(\bar{x}, \nu) \to \R$ such that $\Psi(x) \to 0~(x \to \bar{x})$ and
\begin{align*}
| \P(x) - \P(\bar{x}) | \leq \Psi(x) \Vert \nabla \P(x) \Vert \quad \forall x \in B(\bar{x}, \nu).
\end{align*}
\end{definition}
\noindent
The generalized Polyak-\L ojasiewicz inequality is often assumed in the convergence analysis of optimization methods, as found in~\cite{AnFuHaSaSe21,AnFuHaSaSe24,AnHaVi20,LiYaFu25}. It is generally known that it is satisfied when $G$ is analytic.
\par
From now on, we prepare several lemmas regarding the first- and second-order sequential optimality. These lemmas play a crucial role in the subsequent convergence analysis of the proposed method.
\begin{lemma} \label{lem:x_AKKT}
Let $\{ x_{k} \} \subset \R^{n}$, $\{ \gamma_{k} \} \subset (0,\infty)$, and $\{ \delta_{k} \} \subset (0,1)$ be given sequences, and let $y_{k} \coloneqq - \gamma_{k-1} g(x_{k})$ and $Z_{k} \coloneqq \gamma_{k-1} [-G(x_{k})]_{+}^{3}$ for any $k \in \mathbb{N}$. Suppose that the sequence $\{ x_{k} \}$ converges to a feasible point $\bar{x} \in \mathcal {S}$. Suppose also that $\Vert \F (x_{k+1}; \gamma_{k}) \Vert \leq \delta_{k}$ for all $k \in \mathbb{N} \cup \{ 0 \}$. Then, the point $\bar{x}$ satisfies the AKKT conditions, and $\{ (x_{k}, y_{k}, Z_{k}) \}$ is an AKKT sequence corresponding to $\bar{x}$.
\end{lemma}

\begin{proof}
Let $k \in \mathbb{N}$ be arbitrary. Theorem~\ref{thm:F_derivative_Hessian} derives $\nabla \F(x_{k}; \gamma_{k-1}) = \nabla f(x_{k}) - \nabla g(x_{k}) y_{k} - \D G(x_{k})^{\ast} Z_{k} = \nabla_{x} L(x_{k}, y_{k}, Z_{k})$. It then follows from $\Vert \F (x_{k}; \gamma_{k-1}) \Vert \leq \delta_{k-1}$ that $\nabla L(x_{k}, y_{k}, Z_{k}) \to 0$ as $k \to \infty$.
\par
Next, we denote by $P_{k} \in \R^{d \times d}$ an orthogonal matrix that diagonalizes $G(x_{k}) \in \S^{d}$, that is, $G(x_{k}) = P_{k} {\rm diag}( \mu_{1}^{(k)}, \ldots, \mu_{d}^{(k)} ) P_{k}^{\top}$, where $\mu_{j}^{(k)} \coloneqq \lambda_{j}^{P_{k}}(G(x_{k}))$ for any $j \in \{ 1, \ldots, d \}$. It is clear that $\{ P_{k} \}$ is bounded. Moreover, since $\{ x_{k} \}$ is bounded, so is $\{ \mu_{j}^{(k)} \}$ for each $j \in \{ 1, \ldots, d \}$. Without loss of generality, we can assume that $P_{k} \to P$, $\mu_{1}^{(k)} \to \mu_{1}$, \ldots, $\mu_{d}^{(k)} \to \mu_{d}$ as $k \to \infty$. Note that $P \in \R^{d \times d}$ is an orthogonal matrix satisfying $G(\bar{x}) = P {\rm diag}( \mu_{1}, \ldots, \mu_{d} ) P^{\top}$. Consequently, we have $\lambda^{P}(G(\bar{x})) = [\mu_{1} \, \cdots \, \mu_{d}]^{\top}$. Let $j \in \{ 1, \ldots, d \}$, and suppose that $\lambda_{j}^{P}(G(\bar{x})) > 0$. From $\mu_{j}^{(k)} \to \mu_{j} ~ (k \to \infty)$, there exists $k_{j} \in \mathbb{N}$ such that $|\mu_{j}^{(k)} - \mu_{j}| \leq \frac{1}{2} \mu_{j}$ for all $k \geq k_{j}$. Now, we suppose $k \geq k_{j}$. Then, we obtain $0 < \frac{1}{2} \lambda_{j}^{P}(G(\bar{x})) = \frac{1}{2} \mu_{j} = \mu_{j} - \frac{1}{2} \mu_{j} \leq \mu_{j}^{(k)}$. Hence, the definition of $Z_{k}$ derives $\lambda_{j}^{P_{k}}(Z_{k}) = \gamma_{k-1} [-\lambda_{j}^{P_{k}}(G(x_{k}))]_{+}^{3} = \gamma_{k-1} [-\mu_{j}^{(k)}]_{+}^{3} = 0$. This completes the proof.
\end{proof}

\begin{lemma} \label{lem:x_AKKT2}
Let $\{ x_{k} \} \subset \R^{n}$, $\{ \gamma_{k} \} \subset (0,\infty)$, and $\{ \delta_{k} \} \subset (0,1)$ be given sequences, and let $y_{k} \coloneqq - \gamma_{k-1} g(x_{k})$, $Z_{k} \coloneqq \gamma_{k-1} [-G(x_{k})]_{+}^{3}$, and $\varepsilon_{k} \coloneqq \delta_{k-1}$ for any $k \in \mathbb{N}$. Suppose that the sequence $\{ x_{k} \}$ converges to a feasible point $\bar{x} \in \mathcal {S}$. Suppose also that $\nabla^{2} \F(x_{k+1}; \gamma_{k}) + \delta_{k} I \succeq O$ for all $k \in \mathbb{N}$. Then, there exists a positive integer $\bar{n} \in \mathbb{N}$ such that
\begin{align*}
\forall k \geq \bar{n}, \quad \forall h \in S(x_{k}, \bar{x}), \quad 
h^{\top} \left( \nabla_{xx}^{2}L(x_{k}, y_{k}, Z_{k}) + \sigma(x_{k}, Z_{k}) \right) h \geq - \varepsilon_{k} \Vert h \Vert^{2}.
\end{align*}
\end{lemma}

\begin{proof}
Let $\B \subset \mathbb{N}$, $\N \subset \S^{d}$, and $\U_{\B} \colon \N \to \R^{d \times |\B|}$ be the sets and function defined in Lemma~\ref{lem:BoSh00}, respectively. Now, we define $\A \coloneqq \{ j \in \mathbb{N}; \lambda_{j}(G(\bar{x})) > 0 \}$ and $\ell \coloneqq |\A|$. Note that $\lambda_{1}(G(\bar{x})) \geq \cdots \geq \lambda_{d}(G(\bar{x})) \geq 0$ according to $\bar{x} \in \mathcal{S}$, and recall that $\B = \{ j \in \mathbb{N}; \lambda_{j}(G(\bar{x})) = 0 \}$. Then, we readily have $\A = \{ 1, \ldots, \ell \}$ and $\B = \{ \ell+1, \ldots, d \}$. Recall that $\N$ is a neighborhood of $G(\bar{x})$. The continuity of $G$ implies the existence of $r>0$ satisfying
\begin{align}
G(x) \in \N \quad  \forall x \in B(\bar{x}, r). \label{GinN}
\end{align}
Let $j \in \{ 1, \ldots, \ell \}$ be arbitrary. From the continuity of $G$ and eigenvalues, there exists $r_{j} > 0$ such that $|\lambda_{j}(G(x)) - \lambda_{j}(G(\bar{x}))| \leq \frac{1}{2}\lambda_{j}(G(\bar{x}))$ for all $x \in B(\bar{x}, r_{j})$. It then follows from $\lambda_{j}(G(\bar{x})) > 0$ and $\A = \{ 1,\ldots, \ell \}$ that
\begin{align}
0 < \frac{1}{2} \lambda_{j}(G(\bar{x})) = \lambda_{j}(G(\bar{x})) - \frac{1}{2} \lambda_{j}(G(\bar{x})) \leq \lambda_{j}(G(x)) \quad \forall j \in \A, \, \forall x \in B(\bar{x}, r_{j}). \label{eig_G_positive}
\end{align}
Let us define $\delta \coloneqq \min \{ r, r_{1}, \ldots, r_{\ell} \}$. Since $\{ x_{k} \}$ converges to $\bar{x}$, there exists $\bar{n} \in \mathbb{N}$ such that $x_{k} \in B(\bar{x}, \delta)$ for all $k \geq \bar{n}$. Now, we suppose that $k \geq \bar{n}$. It then follows from~\eqref{GinN}, \eqref{eig_G_positive}, and $x_{k} \in B(\bar{x}, \delta)$ that
\begin{align}
G(x_{k}) \in \N, \quad \lambda_{j}(G(x_{k})) > 0 ~~ \forall j \in \A. \label{GxkinN_positive}
\end{align}
Let $\B_{k}^{+}$, $\B_{k}^{0}$, and $\B_{k}^{-}$ be subsets of $\B$ defined by
\begin{align*}
\B_{k}^{+} \coloneqq \{ j \in \B; \lambda_{j}(G(x_{k})) > 0 \}, \quad \B_{k}^{0} \coloneqq \{ j \in \B; \lambda_{j}(G(x_{k})) = 0 \}, \quad \B_{k}^{-} \coloneqq \{ j \in \B; \lambda_{j}(G(x_{k})) < 0 \},
\end{align*}
respectively. We consider the following matrix decomposition:
\begin{align} \label{decom:Gxk}
\begin{gathered}
G(x_{k}) = P_{k} D_{k} P_{k}^{\top}, ~~ P_{k} = 
\begin{bmatrix}
U_{k} & V_{k}^{+} & V_{k}^{0} & V_{k}^{-}
\end{bmatrix}
, 
\\
D_{k} = {\rm diag}( \lambda_{1}(G(x_{k})), \ldots, \lambda_{d}(G(x_{k})) ) =
\begin{bmatrix}
\Lambda_{k} & O & O & O
\\
O & \Sigma_{k}^{+} & O & O
\\
O & O & O & O
\\
O & O & O & \Sigma_{k}^{-}
\end{bmatrix}
, 
\\
U_{k} \in \R^{d \times \ell}, ~~ V_{k}^{+} \in \R^{d \times |\B_{k}^{+}|}, ~~ V_{k}^{0} \in \R^{d \times |\B_{k}^{0}|}, ~~ V_{k}^{-} \in \R^{d \times |\B_{k}^{-}|},
\\
\Lambda_{k} \in \S_{++}^{\ell}, ~~ \Sigma_{k}^{+} \in \S_{++}^{|\B_{k}^{+}|}, ~~ \Sigma_{k}^{-} \in \S_{--}^{|\B_{k}^{-}|}. 
\end{gathered}
\end{align}
Note that the columns of $[V_{k}^{+} ~ V_{k}^{0} ~ V_{k}^{-}] \in \R^{d \times (d-\ell)}$ belong to the eigenspace for the $d-\ell \, (=|\B|)$ smallest eigenvalues of $G(x_{k})$. Meanwhile, the columns of $\U_{\B}(G(x_{k}))$ form an orthonormal basis of the eigenspace for the $d-\ell$ smallest eigenvalues of $G(x_{k})$. Therefore, there exists an orthogonal matrix $M_{k} \in \R^{(d-\ell) \times (d-\ell)}$ such that 
\begin{align}
V_{k} = \U_{\B}(G(x_{k})) M_{k}. \label{eq:VMUB} 
\end{align}
Let $h \in S(x_{k}, \bar{x})$ be arbitrary. We define 
\begin{align} \label{def:HH}
H_{k} \coloneqq \mbox{D}G(x_{k}) h, \quad \widetilde{H}_{k} \coloneqq P_{k}^{\top} H_{k} P_{k}.
\end{align}
Note that 
\begin{align}
\nabla g(x_{k})^{\top} h = 0, \quad \U_{\B}(G(x_{k}))^{\top} H_{k} \U_{\B}(G(x_{k})) = O. \label{eq:gh_UBH}
\end{align}
It then follows from~\eqref{eq:VMUB} that $V_{k}^{\top} H_{k} V_{k} = M_{k}^{\top} \U_{\B}(G(x_{k}))^{\top} H_{k} \U_{\B}(G(x_{k})) M_{k} = O$, that is,
\begin{align}
\widetilde{H}_{k} = \left[
\begin{array}{cccc}
U_{k}^{\top} H_{k} U_{k} & U_{k}^{\top} H_{k} V_{k}^{+} & U_{k}^{\top} H_{k} V_{k}^{0} & U_{k}^{\top} H_{k} V_{k}^{-}
\\
(V_{k}^{+})^{\top} H_{k} U_{k} & O & O & O
\\
(V_{k}^{0})^{\top} H_{k} U_{k} & O & O & O
\\
(V_{k}^{-})^{\top} H_{k} U_{k} & O & O & O
\end{array}
\right]. \label{eq:tildeH}
\end{align}
From the definition of the sigma term, we obtain
\begin{align}
h^{\top} \sigma(x_{k}, Z_{k}) h 
= 2 \gamma_{k-1} \Big\langle \widetilde{H}_{k}, D_{k}^{\dagger} \widetilde{H}_{k} [-D_{k}]_{+}^{3} \Big\rangle = \gamma_{k-1} \Big\langle \widetilde{H}_{k}, D_{k}^{\dagger} \widetilde{H}_{k} [-D_{k}]_{+}^{3} + [-D_{k}]_{+}^{3} \widetilde{H}_{k} D_{k}^{\dagger} \Big\rangle. \label{eq:dHd}
\end{align}
Moreover, using~\eqref{decom:Gxk} and~\eqref{eq:tildeH} derives
\begin{align} 
D_{k}^{\dagger} \widetilde{H}_{k} [-D_{k}]_{+}^{3} + [-D_{k}]_{+}^{3} \widetilde{H}_{k} D_{k}^{\dagger}
&=
\begin{bmatrix}
O & O & O & -\Lambda_{k}^{-1} U_{k}^{\top}H_{k}V_{k}^{-} (\Sigma_{k}^{-})^{3}
\\
O & O & O & O
\\
O & O & O & O
\\
- (\Sigma_{k}^{-})^{3} (V_{k}^{-} )^{\top} H_{k} U_{k} \Lambda_{k}^{-1} & O & O & O
\end{bmatrix} \nonumber
\\
&=
\begin{bmatrix}
O & O & O & [E_{k}]_{\A \B_{k}^{-}} \circ U_{k}^{\top}H_{k}V_{k}^{-}
\\
O & O & O & O
\\
O & O & O & O
\\
[E_{k}]_{\A \B_{k}^{-}}^{\top} \circ (V_{k}^{-} )^{\top} H_{k} U_{k} & O & O & O
\end{bmatrix}
, \label{eq:DHD}
\end{align}
where the matrix $E_{k} \in \R^{d \times d}$ is defined by 
\begin{align} \label{def:matE}
[E_{k}]_{ij} \coloneqq \left\{
\begin{aligned}
& -\frac{ \lambda_{j}(G(x_{k}))^{3} }{ \lambda_{i}(G(x_{k})) } && \mbox{if}~i \in \A, \, j \in \B_{k}^{-},
\\
& 0 && \mbox{otherwise}.
\end{aligned}
\right.
\end{align}
On the other hand, noting~\eqref{def:HH} leads to
\begin{align}
\gamma_{k-1} \big\langle H_{k}, \D \mathcal{Q}(-G(x_{k}))H_{k} \big\rangle = \gamma_{k-1} \Big\langle \widetilde{H}_{k}, P_{k}^{\top} (\D \mathcal{Q}(-G(x_{k}))H_{k}) P_{k} \Big\rangle. \label{eq:rhoHDQH}
\end{align}
Lemma~\ref{lem:derivative_Q} and equation~\eqref{eq:tildeH} imply
\begin{align} 
P_{k}^{\top} \D \mathcal{Q}(-G(x_{k}))H_{k} P_{k} = 
\begin{bmatrix}
O & O & O & [F_{k}]_{\A \B_{k}^{-}} \circ U_{k}^{\top} H_{k} V_{k}^{-}
\\
O & O & O & O
\\
O & O & O & O
\\
[F_{k}]_{\A \B_{k}^{-}}^{\top} \circ (V_{k}^{-})^{\top} H_{k} U_{k} & O & O & O
\end{bmatrix}
, \label{derivative:QH}
\end{align}
where the matrix $F \in \R^{d \times d}$ is defined as
\begin{align} \label{def:matF}
[F_{k}]_{ij} &\coloneqq \left\{
\begin{aligned}
& \frac{\lambda_{j}(G(x_{k}))^{3}}{\lambda_{j}(G(x_{k})) - \lambda_{i}(G(x_{k}))} && \mbox{if}~i \in \A, \, j \in \B_{k}^{-},
\\
& 0 && \mbox{otherwise}.
\end{aligned}
\right.
\end{align}
By~\eqref{eq:DHD} and~\eqref{derivative:QH}, we have
\begin{align*}
& D_{k}^{\dagger} \widetilde{H}_{k} [-D_{k}]_{+}^{3} + [-D_{k}]_{+}^{3} \widetilde{H}_{k} D_{k}^{\dagger} - P_{k}^{\top} \D \mathcal{Q}(-G(x_{k}))H_{k} P_{k} 
\\
&= 
\begin{bmatrix}
O & O & O & [(E_{k} - F_{k}) \circ \widetilde{H}_{k}]_{\A \B_{k}^{-}}
\\
O & O & O & O
\\
O & O & O & O
\\
[(E_{k} - F_{k}) \circ \widetilde{H}_{k}]_{\A \B_{k}^{-}}^{\top} & O & O & O
\end{bmatrix}
.
\end{align*}
It then follows from~\eqref{eq:dHd} and~\eqref{eq:rhoHDQH} that
\begin{align}
& h^{\top} \sigma(x_{k}, Z_{k}) h - \gamma_{k-1} \big\langle H_{k}, \D \mathcal{Q}(-G(x_{k}))H_{k} \big\rangle \nonumber
\\
&= \gamma_{k-1} \Big\langle \widetilde{H}_{k}, D_{k}^{\dagger} \widetilde{H}_{k} [-D_{k}]_{+}^{3} + [-D_{k}]_{+}^{3} \widetilde{H}_{k} D_{k}^{\dagger} - P_{k}^{\top} (\D \mathcal{Q}(-G(x_{k}))H_{k}) P_{k} \Big\rangle \nonumber
\\
&= 2 \gamma_{k-1} \sum_{i \in \A} \sum_{j \in \B_{k}^{-}} [E_{k} - F_{k}]_{ij} [\widetilde{H}_{k}]_{ij}^{2}. \label{eq:DHD_PDP}
\end{align}
Noting~\eqref{def:matE} and~\eqref{def:matF} yields
\begin{align*}
[E_{k} - F_{k}]_{ij} = \frac{ \lambda_{j}(G(x_{k}))^{4} }{ \lambda_{i}(G(x_{k}))^{2} - \lambda_{i}(G(x_{k})) \lambda_{j}(G(x_{k})) } \geq 0 \quad \forall i \in \A, \, \forall j \in \B_{k}^{-}.
\end{align*}
Hence, equations~\eqref{def:HH} and~\eqref{eq:DHD_PDP} ensure
\begin{align}
h^{\top} \sigma(x_{k}, Z_{k}) h - \gamma_{k-1} \big\langle \D G(x_{k}) h, \D \mathcal{Q}(-G(x_{k})) \D G(x_{k}) h \big\rangle \geq 0. \label{ineq:sigma_HQH}
\end{align}
Now, Theorem~\ref{thm:F_derivative_Hessian} derives that
\begin{align*}
h^{\top} \nabla^{2} \F(x_{k}; \gamma_{k-1}) h = h^{\top} \nabla_{xx}^{2} L(x_{k}, y_{k}, Z_{k}) h + \gamma_{k-1} \Vert \nabla g(x_{k})^{\top} h \Vert^{2} + \gamma_{k-1} \big\langle \D G(x_{k}) h, \D \mathcal{Q}(-G(x_{k})) \D G(x_{k}) h \big\rangle.
\end{align*}
Combining this equality, $h \in S(x_{k}, \bar{x})$, and $\nabla^{2} \F(x_{k}; \gamma_{k-1}) + \varepsilon_{k} I = \nabla^{2} \F(x_{k}; \gamma_{k-1}) + \delta_{k-1} I \succeq O$ yields
\begin{align}
h^{\top} \left( \nabla_{xx}^{2} L(x_{k}, y_{k}, Z_{k}) + \sigma(x_{k}, Z_{k} \right) h \geq  h^{\top} \sigma(x_{k}, Z_{k}) h - \gamma_{k-1} \big\langle \D G(x_{k}) h, \D \mathcal{Q}(-G(x_{k})) \D G(x_{k}) h \big\rangle - \varepsilon_{k} \Vert h \Vert^{2}. \label{ineq:HessL_sigma}
\end{align}
From~\eqref{ineq:sigma_HQH} and~\eqref{ineq:HessL_sigma}, the assertion is proven.
\end{proof}

From now on, we present a penalty method for finding a point that satisfies the second-order sequential optimality conditions. 
Let $\{ \gamma_{k} \} \subset (0,\infty)$ be a penalty parameter sequence satisfying $\gamma_{k} \to \infty~(k \to \infty)$. The proposed penalty method iteratively solves the following subproblem at each iteration $k \in \mathbb{N} \cup \{ 0 \}$:
\begin{align}
\begin{aligned} \label{subpro:penalty}
& \mini_{x \in \R^{n}} && \F(x; \gamma_{k}) = f(x) + \frac{\gamma_{k}}{2} \Vert g(x) \Vert^{2} + \frac{\gamma_{k}}{4} {\rm tr}([-G(x)]_{+}^{4}),
\end{aligned}
\end{align}
In particular, we suppose that problem~\eqref{subpro:penalty} is solved by utilizing existing trust region methods, such as~\cite[Chapter~4]{NoWr06}. Then, the formal statement of the proposed penalty method is provided as Algorithm~\ref{algo:penalty}.

\begin{algorithm} 
\caption{~(Penalty method)} \label{algo:penalty}
\begin{algorithmic}[1]

\State Choose $\eta \in (0,1)$ and $\theta > 1$. Select $\{ \delta_{k} \} \subset (0,1)$ satisfying $\delta_{k} \to 0~(k \to \infty)$. Set $x_{0} \in \mathcal{S}$, $\widehat{x}_{0} \coloneqq x_{0}$, $\gamma_{0} > 0$, and $k \coloneqq 0$.

\State Solve problem~\eqref{subpro:penalty} by using a trust region method with an initial point $\widehat{x}_{k}$, and obtain its approximate solution $x_{k+1}$ satisfying 
\begin{equation}
\Vert \nabla \F (x_{k+1}; \gamma_{k}) \Vert \leq \delta_{k} , \quad \nabla^{2} \F(x_{k+1}; \gamma_{k}) + \delta_{k} I \succeq O. \label{penalty_condition}
\end{equation}

\State Update $\widehat{x}_{k}$ and $\gamma_{k}$ as follows:
\begin{align} \label{rule:x_hat}
\widehat{x}_{k+1} \coloneqq
\begin{cases}
x_{k+1} & {\rm if} ~ \F(x_{k+1}; \gamma_{k}) \leq f(x_{0}),
\\
x_{0} & {\rm otherwise,}
\end{cases}
\quad \gamma_{k+1} \coloneqq
\begin{cases}
\gamma_{k} & {\rm if} ~ k=0 ~ {\rm or} ~ u_{k+1} \leq \eta u_{k},
\\
\theta \gamma_{k} & {\rm otherwise,}
\end{cases}
\end{align}
where $u_{k} \coloneqq \max \{ \Vert g(x_{k}) \Vert, \, \Vert [-G(x_{k})]_{+} \Vert_{{\rm F}} \}$.

\State Set $k \leftarrow k + 1$ and go back to Step~2.

\end{algorithmic}
\end{algorithm}

\begin{remark}
A penalty method for finding AKKT2 and CAKKT2 points has already been proposed in~\cite{LiYaFu25}. The main difference between Algorithm~\ref{algo:penalty} and the existing method is the subproblems minimizing their penalty functions. Since the penalty function used in Algorithm~\ref{algo:penalty} is twice continuously differentiable, we can utilize existing trust region methods equipped with global convergence to second-order stationary points, such as~\cite[Chapter~4]{NoWr06}, to obtain an approximate solution of~\eqref{subpro:penalty} satisfying~\eqref{penalty_condition}. Meanwhile, at each iteration, the existing method also needs to solve the subproblem minimizing the penalty function and find its stationary point satisfying some additional conditions similar to its second-order optimality. However, as mentioned in~\cite[Remark~5.2]{LiYaFu25}, the existing method has a practical issue because a way of finding such a stationary point is an open question. The issue arises from the fact that the penalty function used in~\cite{LiYaFu25} is not twice continuously differentiable. Indeed, due to the differentiability of the penalty function, it is difficult for the existing method to adopt a similar way to Algorithm~\ref{algo:penalty}. Moreover, Algorithm~\ref{algo:penalty} does not require the penalty parameter to diverge to infinity, but the existing method has to do it. Therefore, Algorithm~\ref{algo:penalty} can be regarded as a revised version of the existing penalty method of~\cite{LiYaFu25}.
\end{remark}

\noindent
We show that Algorithm~\ref{algo:penalty} can find AKKT2 and CAKKT2 points for problem~\eqref{NSDP}. To this end, we prepare the following lemma.

\begin{lemma} \label{lem:feasible}
Let $\{ x_{k} \} \subset \R^{n}$ be a sequence generated by Algorithm~{\rm \ref{algo:penalty}}. If $\{ x_{k} \}$ converges to a point $\bar{x}$, then it satisfies $\bar{x} \in \mathcal{S}$.
\end{lemma}

\begin{proof}
Let $\I \coloneqq \{ k \in \mathbb{N} \cup \{ 0 \}; \gamma_{k+1} = \gamma_{k} \}$ and $\J \coloneqq \{ k \in \mathbb{N} \cup \{ 0 \}; \gamma_{k+1} = \theta \gamma_{k} \}$. There are two possible cases: (a) $|\J| < \infty$; (b) $|\J| = \infty$. We begin by considering case~(a). In this case, there eixsts $\bar{m} \in \mathbb{N} \backslash \{ 0 \}$ such that $k \in |\I|$ for all $k \geq \bar{m}$. For an arbitrary integer $j > 0$, we have $u_{\bar{m}+j} \leq \eta u_{\bar{m}+j-1} \leq \eta^{2} u_{\bar{m}+j-2} \leq \cdots \leq \eta^{j} u_{\bar{m}}$ from the second definition of~\eqref{rule:x_hat} and $\bar{m} + j, \bar{m} + j -1, \ldots, \bar{m} \in |\J|$. This fact and $\eta \in (0,1)$ imply $u_{k} \to 0$ as $k \to \infty$, that is, $\bar{x} \in \mathcal{S}$. Next, we consider case~(b). This case ensures that $\gamma_{k} \to \infty ~ (k \to \infty)$. Let $k \in \mathbb{N}$ be an arbitrary integer. From~\eqref{rule:x_hat}, there are two cases: $\widehat{x}_{k} = x_{k}$ and $\widehat{x}_{k} = x_{0}$. The former case yields $\F(\widehat{x}_{k}, \gamma_{k-1}) = \F(x_{k}; \gamma_{k-1}) \leq f(x_{0})$. In the latter case, the feasibility of $x_{0}$ derives $\F(\widehat{x}_{k}; \gamma_{k-1}) = \F(x_{0}; \gamma_{k-1}) = f(x_{0})$. Consequently, we have $\F(\widehat{x}_{k}; \gamma_{k-1}) \leq f(x_{0})$. Meanwhile, we recall that $x_{k}$ is obtained by the trust region method with an initial point $\widehat{x}_{k-1}$. It then follows from the feature of the trust region method that $\F(x_{k}; \gamma_{k-1}) \leq \F(\widehat{x}_{k-1}; \gamma_{k-1})$. These facts ensure that $\F(x_{k}; \gamma_{k-1}) \leq f(x_{0})$, and hence we have
$\frac{1}{2} \Vert g(x_{k}) \Vert^{2} + \frac{1}{4} \Vert [-G(x_{k})]_{+}^{2} \Vert_{{\rm F}}^{2} \leq \frac{1}{\gamma_{k-1}} (f(x_{0}) - f(x_{k}))$. 
Thus, taking $k \to \infty$ yields $g(\bar{x}) = 0$ and $G(\bar{x}) \succeq O$. As a result, both cases~(a) and~(b) lead to $\bar{x} \in \mathcal{S}$.
\end{proof}

\noindent
Finally, we provide the global convergence of Algorithm~\ref{algo:penalty}.
\begin{theorem}
Let $\{ x_{k} \} \subset \R^{n}$ be a sequence generated by Algorithm~{\rm \ref{algo:penalty}}, and let $\{ y_{k} \} \subset \R^{m}$ and $\{ Z_{k} \} \subset \S_{+}^{d}$ be defined by $y_{k} \coloneqq - \gamma_{k-1} g(x_{k})$, $Z_{k} \coloneqq \gamma_{k-1} [-G(x_{k})]_{+}^{3}$, and $\varepsilon_{k} \coloneqq \delta_{k-1}$ for all $k \in \mathbb{N}$. Suppose that $\{ x_{k} \}$ has an accumulation point $\bar{x}$. Then, the following statements hold:
\begin{itemize}
\item[{\rm (i)}] $\bar{x}$ is an AKKT{\rm 2} point, and $\{ (x_{k}, y_{k}, Z_{k}, \varepsilon_{k}) \}$ is an AKKT{\rm 2} sequence corresponding to $\bar{x}$;
\item[{\rm (ii)}] if $\bar{x}$ satisfies the Polyak-\L ojasiewicz inequality, then $\bar{x}$ is a CAKKT{\rm 2} point, and $\{ (x_{k}, y_{k}, Z_{k}, \varepsilon_{k}) \}$ is a CAKKT{\rm 2} sequence corresponding to $\bar{x}$.
\end{itemize}
\end{theorem}

\begin{proof}
Without loss of generality, we assume that $x_{k} \to \bar{x}$ as $k \to \infty$. Now, we show item~(i). It follows from Lemma~\ref{lem:feasible} that $\bar{x} \in \mathcal{S}$. Thus, Lemma~\ref{lem:x_AKKT} ensures that $\bar{x}$ satisfies the AKKT conditions, and $\{ (x_{k}, y_{k}, Z_{k}) \}$ is an AKKT sequence corresponding to $\bar{x}$. Therefore, combining these facts and Lemma~\ref{lem:x_AKKT2} derives the assertion of item~(i).
\par
Next, we prove item~(ii). From the assertion of item~(i), it is sufficient to show that $G(x_{k}) \circledcirc Z_{k} \to O$ as $k \to \infty$. Since $\bar{x}$ is feasible and satisfies the Polyak-\L ojasiewicz inequality, there exist $\nu > 0$ and $\Psi \colon B(\bar{x}, \nu) \to \R$ such that $\Psi(x) \to 0 ~ (x \to \bar{x})$ and
\begin{align} \label{ineq:PL}
\frac{1}{2} \Vert g(x) \Vert^{2} + \frac{1}{4} {\rm tr} ( [-G(x)]_{+}^{4} ) \leq \Psi(x) \Vert \nabla g(x) g(x) - \D G(x)^{\ast} [-G(x)]_{+}^{3} \Vert \quad \forall x \in B(\bar{x}, \nu).
\end{align} 
From $x_{k} \to \bar{x}~(k \to \infty)$, there exists $\bar{n} \in \mathbb{N}$ such that $x_{k} \in B(\bar{x}, \nu)$ for all $k \geq \bar{n}$. Now, we supppose that $k \geq \bar{n}$. Recall that $\nabla \F (x_{k}; \gamma_{k-1}) = \nabla f(x_{k}) + \gamma_{k-1} \nabla g(x_{k}) g(x_{k}) - \gamma_{k-1} \D G(x_{k})^{\ast} [-G(x_{k})]_{+}^{3}$. It then follows from $x_{k} \in B(\bar{x}, \nu)$, $\Vert \nabla \F(x_{k}, \gamma_{k-1}) \Vert \leq \delta_{k}$, and~\eqref{ineq:PL} that
\begin{align} \label{ineq:G_Psi}
{\rm tr}( \gamma_{k-1} [-G(x_{k})]_{+}^{4}) 
\leq 4 \Psi(x_{k}) \Vert \nabla \F (x_{k}; \gamma_{k-1}) - \nabla f(x_{k}) \Vert
\leq 4 \Psi(x_{k}) \left( \delta_{k-1} + \Vert \nabla f(x_{k}) \Vert \right).
\end{align}
Meanwhile, we notice that $[-G(x_{k})]_{+}^{2} = (-G(x_{k})) [-G(x_{k})]_{+}$ and $Z_{k} = \gamma_{k-1} [-G(x_{k})]_{+}^{3}$. These equalities lead to
\begin{align}
\Vert G(x_{k}) \circledcirc Z_{k} \Vert_{{\rm F}} \leq \Vert G(x_{k}) Z_{k} \Vert_{{\rm F}} = \sqrt{ {\rm tr} \left( \left( \gamma_{k-1} [-G(x_{k})]_{+}^{4} \right)^{2} \right) }
\leq {\rm tr}\left( \gamma_{k-1} [-G(x_{k})]_{+}^{4} \right), \label{ineq:GcircledcircZ}
\end{align}
where the last inequality is derived from ${\rm tr}(M^{2}) \leq {\rm tr}(M)^{2}$ for any $M \succeq O$. Combining~\eqref{ineq:G_Psi} and~\eqref{ineq:GcircledcircZ} implies $\Vert G(x_{k}) \circledcirc Z_{k} \Vert_{{\rm F}} \leq 4 \Psi(x_{k}) \left( \delta_{k-1} + \Vert \nabla f(x_{k}) \Vert \right)$. It then follows from $\Psi(x_{k}) \to 0~(k \to \infty)$ that $G(x_{k}) \circledcirc Z_{k} \to O$ as $k \to \infty$. Therefore, the proof is completed.
\end{proof}

\section{Conclusion} \label{sec:conclusion}
We have proposed a twice continuously differentiable penalty function related to problem~\eqref{NSDP}. Some existing optimization methods, such as penalty methods and augmented Lagrangian methods, utilize a penalty function to ensure their convergence property, and several types of penalty functions have been proposed. However, the existing penalty functions have low compatibility with optimization methods to find second-order stationary points because they are not twice continuously differentiable. Meanwhile, the proposed penalty function is expected to have a high affinity for such methods because of its high-order differentiability. To indicate the high affinity, a practical penalty method to find AKKT2 and CAKKT2 points has been proposed by incorporating the proposed function and an existing trust region method, and its global convergence to the second-order stationary points has been proven. As future research, it is worth proposing other optimization methods, such as augmented Lagrangian and SQP-type methods, equipped with the proposed penalty method, and showing their global convergence to second-order stationary points.



\bibliographystyle{abbrv}
\bibliography{references}


\end{document}